\documentclass{amsart}

\usepackage{enumerate,fullpage,mathtools,mathptmx,xcolor,bbold,amssymb,amsmath}
\usepackage[colorlinks=true,linkcolor=red!70!black,citecolor=green!70!black,urlcolor=magenta!70!black]{hyperref}

\newtheorem{thm}{Theorem}
\newtheorem{lem}[thm]{Lemma}

\newtheorem{con}[thm]{Conjecture}

\newcommand{\defeq}{\stackrel{\mathrm{def}}{=}}
\newcommand{\B}{\mathcal{B}}
\newcommand{\C}{\mathcal{C}}
\newcommand{\D}{\mathcal{D}}
\newcommand{\ZZ}{\mathbb{Z}}
\newcommand{\RR}{\mathbb{R}}
\newcommand{\QQ}{\mathbb{Q}}
\newcommand{\CC}{\mathbb{C}}
\newcommand{\oone}{\mathbb{1}}
\newcommand{\p}{\mathsf{P}}
\newcommand{\NP}{\mathsf{NP}}
\newcommand{\shP}{\#\mathsf{P}}
\newcommand{\CSP}{\mathsf{CSP}}
\newcommand{\BQP}{\mathsf{BQP}}
\newcommand{\F}{\mathcal{F}}
\newcommand{\x}{{\mathbf{x}}}
\newcommand{\N}{\mathcal{N}}

\DeclareMathOperator{\Irr}{Irr}

\DeclareMathOperator{\End}{End}

\usepackage{tikz}
\usetikzlibrary{decorations.markings}

\begin{document}


\title{Towards a complexity-theoretic dichotomy for TQFT invariants}

\author{Nicolas Bridges}
\address{Department of Mathematics \\ 150 N University St \\ Purdue University \\ West Lafayette, IN \\ 47907}
\email{bridge18@purdue.edu}

\author{Eric Samperton}
\address{Departments of Mathematics and Computer Science \\ 150 N University St \\ Purdue University \\ West Lafayette, IN \\ 47907}
\email{eric@purdue.edu}

\thanks{Both authors supported in part by NSF CCF grant 2330130.}

\date{\today}

\begin{abstract}
We show that for any fixed $(2+1)$-dimensional TQFT over $\mathbb{C}$ of either Turaev-Viro-Barrett-Westbury or Reshetikhin-Turaev type, the problem of (exactly) computing its invariants on closed 3-manifolds is either solvable in polynomial time, or else it is $\#\mathsf{P}$-hard to (exactly) contract certain tensors that are built from the TQFT's fusion category.
Our proof is an application of a dichotomy result of Cai and Chen [J.~ACM, 2017] concerning weighted constraint satisfaction problems over $\mathbb{C}$.
We leave for future work the issue of reinterpreting the conditions of Cai and Chen that distinguish between the two cases (\emph{i.e.}~$\#\mathsf{P}$-hard tensor contractions vs.~polynomial time invariants) in terms of fusion categories.
We expect that with more effort, our reduction can be improved so that one gets a dichotomy directly for TQFTs' invariants of 3-manifolds rather than more general tensors built from the TQFT's fusion category.

\end{abstract}

\maketitle

\section{Introduction}
\label{sec:intro}
\subsection{Main results}
\label{ss:results}
Quantum computation---especially \emph{topological} quantum computation---motivates a number of complexity-theoretic questions concerning TQFT invariants of manifolds, particularly in dimensions 2 and 3.
One of the most central is to classify ``anyonic systems" according to whether or not they are powerful enough to (approximately) encode arbitrary quantum circuits over qubits.
Anyons that are powerful in this way are important because (in theory) it should be possible to build fault tolerant quantum computers using them \cite{Kitaev,FreedmanLarsenWang-universal}.
We refer the reader to \cite{RowellWang-review} for a broad review of the mathematical side of these matters and Subsection \ref{ss:implications} for more discussion.
For now, we simply note that the Property F conjecture of Naidu and Rowell is currently the only concrete, published formulation of a proposed (partial) answer to this classification question that we know.
The conjecture is surprisingly easy to formulate: the possible braidings of $n$ copies of a simple anyon $X$ in a unitary modular tensor category $\B$ generate  only finitely many unitaries for each $n$ (and, hence, are not ``braiding universal" for quantum computation) if and only if the square of the quantum dimension of $X$ is an integer $d_X^2 \in\ZZ$ \cite{NaiduRowell-F}.

In this work, we will not attack the Property F conjecture directly.
However, our main result has a similar spirit and shares the same motivations.
See Subsection \ref{ss:implications} for some discussion.

\begin{thm}\label{thm:main}
\begin{enumerate}[(a)]
\item Fix a spherical fusion category $\C$ over $\CC$, presented skeletally with all data given as algebraic numbers over $\QQ$.
Then $\#\CSP(\F_\C)$--the problem of contracting tensor networks defined from $\C$---is either solvable in polynomial time or $\shP$-hard.
Moreover, if $M$ is a closed, oriented, triangulated $3$-manifold (treated as computational input), then either the problem of computing the Turaev-Viro-Barrett-Westbury invariant $|M|_{\C} \in \CC$ is solvable in (classical) polynomial time or $\#\CSP(\F_\C)$ is $\shP$-hard.
\item Fix a modular fusion category $\B$ over $\CC$, presented skeletally with all data given as algebraic numbers over $\QQ$.
Then $\#\CSP(\F_\B)$--the problem of contracting tensor networks built from $\B$---is either solvable in polynomial time or $\shP$-hard.
Moreover, if $M$ is a closed, oriented 3-manifold encoded via a surgery diagram (treated as computational input), then either the problem of computing the Reshetikhin-Turaev invariant $\tau_{\B}(M) \in \CC$ is solvable in (classical) polynomial time or $\#\CSP(\F_\B)$ is $\shP$-hard. 
\end{enumerate}
\end{thm}

Two routine points of clarification are due.

First, we note that all fusion and modular categories over $\CC$ admit finite skeletal presentations using algebraic numbers over $\QQ$.
This is because the defining equations for the skeletal data are all algebraic over $\QQ$.
In particular, since we are interested in how the complexity of $|M|_\C$ or $\tau_\B(M)$ depends on variable $M$ for \emph{fixed} $\C$ or $\B$, there is no harm in assuming that $\C$ and $\B$ are encoded in this way.

Second, as usual in computational 3-manifold topology, to say that a problem whose input is a triangulated 3-manifold is solvable in polynomial time means that there exists an algorithm to solve the problem that runs in time polynomial in the size of the triangulation.
For a 3-manifold presented via integral surgery on a link diagram in $S^3$, the algorithm must run in time jointly polynomial in the crossing number of the link diagram, the number of components of the link, and the absolute values of the surgery coefficients.\footnote{In particular, we might understand the surgery coefficients as being expressed in unary, not binary.  (If we used the latter, then it would not be clear how to build a triangulation from a surgery diagram in polynomial time.)}

We refer the reader to \cite{me-TY} for further elaboration of both of these matters.

We now explain briefly the meaning and importance of dichotomy theorems within complexity theory.
Of course, it is an infamous open problem to show that $\p \ne \NP$ (the two complexity classes might be equal, but most experts do not expect this to be the case).
To establish this inequality it is necessary and sufficient to show that there exists an $\NP$-complete problem with no polynomial-time algorithm.  
Intriguingly, Ladner showed that if $\p \ne \NP$, then there exist problems in $\NP$ that are neither in $\p$ nor $\NP$-complete \cite{Ladner}.
These are usally referred to as ``$\NP$-intermediate."
In other words, an $\NP$-intermediate problem is a problem in $\NP$ that is neither in $\p$ nor $\NP$-hard.
Intuitively, Ladner's theorem shows that if one considers a family of decision problems, then it need not be the case that every problem in the family is either ``easy" (that is, in $\p$) or ``hard" (that is, $\NP$-hard)--there could be problems that have intermediate complexity.
When a given family of problems has the property that none of the problems has intermediate complexity, then one says that the family satisfies a ``dichotomy theorem."
The archetypical dichotomy theorem was established by Schaefer, who showed that ``local" Boolean satisfiability problems (parametrized by a set of allowed ``local constraints") satisfy a dichotomy theorem \cite{Schaefer}.

In the case of our Theorem \ref{thm:main}, we interpret (2+1)-d TQFT invariants as generalized (\emph{i.e.}\ $\CC$-valued instead of $\mathbb{N}$-valued) counting problems parametrized by spherical fusion categories $\C$ and modular fusion categories $\B$ (the categories are analogs of the allowed local constraints in Schaefer's dichotomy).
Our results establish the dichotomy that either a function of the type $M \mapsto |M|_\C \in \CC$ or $M \mapsto \tau_\B(M) \in \CC$ is ``easy" to compute (polynomial time) or else it is ``hard" to contract certain tensors built from the category $\C$ or $\B$ ($\shP$-hard).
In this way, there are no (2+1)-d TQFTs whose tensors are of ``intermediate" complexity.
In fact, we conjecture that the same can be said directly of the TQFT's \emph{invariants of 3-manifolds per se}.
Let us expound on these points now.

\subsection{Mapping the dichotomy}
\label{ss:dichotomy}
Whether or not $M \mapsto |M|_\C$ or $M \mapsto \tau_\B(M)$ is computable in polynomial time depends on $\C$ and $\B$. 
Having established Theorem \ref{thm:main}---which only asserts the \emph{existence} of a dichotomy---it is natural to wonder where one should draw the line between easy and hard.  Better yet, ideally, one would like to be able to prove that the dichotomy of Theorem \ref{thm:main} is \emph{effective}, meaning, given $\C$ or $\B$, there exists a polynomial-time algorithm to decide precisely when the category falls into the easy case (here ``polynomial-time" means in the size of the skeletalization of $\C$ or $\B$).
Our proof of Theorem \ref{thm:main} relies on the main result of Cai and Chen's work \cite{CaiChen-dichotomy}, which establishes a dichotomy theorem for a generalized type of ``solution counting" to constraint satisfaction problems $\#\CSP(\F)$ with a fixed ``$\CC$-weighted constraint family" $\F$.
We carefully define $\#\CSP(\F)$ in Subsection \ref{ss:CaiChen}, but here we note that not only do Cai and Chen prove that for every choice of constraint family $\F$, $\#\CSP(\F)$ is either $\shP$-hard or computable in polynomial time---they also provide three necessary and sufficient conditions that \emph{characterize precisely} which $\F$ allow for polynomial time solutions to $\#\CSP(\F)$.
These conditions are called ``block orthogonality," ``Mal'tsev" and ``Type Partition".
Our proof of Theorem \ref{thm:main} consists in converting a spherical fusion category $\C$ or modular fusion category $\B$ into an appropriate constraint family $\F_\C$ or $\F_\B$ such that computing $|M|_\C$ or $\tau_\B(M)$ is equivalent to computing an instance (depending on $M$) of a problem in $\#\CSP(\F_\C)$ or $\#\CSP(\F_\B)$, respectively.
In particular, for the constraint families $\F_\C$ and $\F_\B$ we shall build, it should be possible to interpret the three conditions of Cai and Chen directly in terms of the categories $\C$ and $\B$.
It is beyond the scope of the present work to attempt to accomplish this.
However, we believe this is an important problem, since it should shed light on variations of the Property F conjecture related to anyon classification, as we explain in the next subsection.

Let us now address the more important deficiency of Theorem \ref{thm:main}, alluded to at the end of the previous subsection: it would be better to get an outright dichotomy for 3-manifold invariants, and not just general tensors derived from a fusion category.  To this end, Theorem \ref{thm:main} can be understood as a first step towards proving the following more desirable result.

\begin{con}\label{con:main}
\begin{enumerate}[(a)]
\item Fix a spherical fusion category $\C$ over $\CC$, presented skeletally with all data given as algebraic numbers over $\QQ$.
If $M$ is a closed, oriented, triangulated $3$-manifold (treated as computational input), then computing the Turaev-Viro-Barrett-Westbury invariant $|M|_{\C} \in \CC$ is either solvable in (classical) polynomial time or is $\shP$-hard.
\item Fix a modular fusion category $\B$ over $\CC$, presented skeletally with all data given as algebraic numbers over $\QQ$.
If $M$ is a closed, oriented 3-manifold encoded via a surgery diagram (treated as computational input), then the problem of computing the Reshetikhin-Turaev invariant $\tau_{\B}(M) \in \CC$ is either solvable in (classical) polynomial time or is $\shP$-hard. 
\end{enumerate}
\end{con}
See Subsection \ref{ss:improvements} for some discussion of how we expect one might get started on proving this conjecture---in particular, the relevance of holant problems \cite{CaiETAL-holantdef}.

\subsection{Implications for ``anyon classification"}
\label{ss:implications}
Theorem \ref{thm:main} and Conjecture \ref{thm:main} assert that for certain precise formulations of the problem of ``anyon classification," whatever the ``type" is for a given modular fusion category $\B$, it can only be one of two things, with no ``intermediate" cases.
In order to explain this more carefully, we pause to note the many ways one can make the problem of anyon classification precise, and situate our result exactly in this milieu.
Moving from the more ``purely mathematical" to the more ``applied" end of the spectrum, ``anyon classification" could mean any of the following precise problems:
\begin{enumerate}
\item Algebraic classification of unitary modular fusion categories (MFCs) up to ribbon tensor equivalence.
	\begin{itemize}
	\item Much of the literature on fusion categories can be considered as contributing to this problem.
	\item Presumably one would be satisfied with a solution to this problem ``modulo finite group theory."
	\end{itemize}
\item Algebraic classification of simple objects $X$ in unitary MFCs $\B$ according to whether or not the braid group representations $B_n \to U(\End_\B(X^{\otimes n}))$ have finite image, dense image, or something else.
	\begin{itemize}
	\item The Property F conjecture is of course directly related to this matter.
	\item One can generalize this question to consider mapping class group representations of higher genus surfaces with different types of anyons on them.
	\end{itemize}
\item Complexity-theoretic classification of MFCs according to how easy or hard it is to \emph{exactly compute} their Reshetikhin-Turaev 3-manifold invariants (as algebraic numbers over $\QQ$).
	\begin{itemize}
		\item {\bf Our Theorem \ref{thm:main} is situated here--almost!}  We have established a dichotomy of the form either ``3-manifold invariants easy" or ``tensors in the category are hard to contract".  Our results represent a non-trivial step towards the desired dichotomy of Conjecture \ref{con:main}: ``invariants easy" or ``invariants hard".
	\item One might ask this question for \emph{restricted classes of 3-manifolds} (such as ``knots in $S^3$" or ``links in $S^3$" or ``integer homology 3-spheres"), and the classification might change \cite{me-TY}.
	\end{itemize}
\item Complexity-theoretic classification of MFCs according to how easy or hard it is to ``\emph{approximate}" their Reshetikhin-Turaev 3-manifold invariants.
	\begin{itemize}
	\item There are different types of approximations one might ask for.  \emph{A priori}, each type should be understood as giving a different version of this question.
	\item ``Exactly compute" is one way to ``approximate."
	\item Pioneering works of Freedman, Kitaev, Larsen and Wang show that for certain approximation schemes, there exists unitary MFCs $\B$ for which the ability to approximate their 3-manifold invariants is equivalent in power to $\BQP$ (bounded-error quantum polynomial time) \cite{FreedmanKitaevWang,FreedmanLarsenWang-universal}.  In particular, their work established the original paradigm for topological quantum computation via anyon braiding.
	\item Kuperberg showed that results for one type of approximation can have important implications for other types of approximations \cite{Kuperberg-hard}.  In particular, the kinds of approximations that a quantum computer can efficiently make for Reshetikhin-Turaev invariants are (in general/worst case) not precise enough to do anything useful for distinguishing 3-manifolds even if their invariants are promised to be unequal by a large amount.
	\end{itemize}
\item Complexity-theoretic classification of MFCs according to whether or not they support universal quantum computation with braiding \emph{and adaptive anyonic charge measurements}. 
	\begin{itemize}
	\item Quantum computation using braidings and charge measurements of anyons in a fixed unitary MFC is often called ``topological quantum computing with adaptive charge measurement."
	\item Our entirely subjective opinion is that this is the most important flavor of anyon classification, at least when considered from the perspective of the goal of actually building a universal, fault-tolerant quantum computer.
	\item Even unitary MFCs whose anyons all have Property F (and, hence, are not universal via braiding alone) can be universal when braiding is supplemented with charge measurements, see e.g.\ \cite{CuiWang-metaplectic}.
	\item While adaptive charge measurement is generally considered fault-tolerant for topological reasons, unlike the case of braiding-only topological quantum computing, the amplitudes with which one performs a quantum computation in this paradigm are not (normalizations of) Reshetikhin-Turaev invariants of 3-manifolds. 
	\end{itemize}
\end{enumerate}
There are known relations between these different classification problems.  For example, on one hand, if $X$ is an anyon such that $B_n \to U(\End_\B(X^{\otimes n}))$ is dense, then the Solovay-Kitaev theorem implies that $\B$ supports universal topological quantum computation via braiding (without needing adaptive charge measurement).  On the other hand, if $X$ has Property F, then it is known that braiding with $X$ is never powerful enough to encode all of $\BQP$ in its braidings.  This latter point was the main motivation for the Property F conjecture in the first place, since one would like to rule out the ``obviously" un-useful anyons easily. 

To understand the potential usefulness of Theorem \ref{thm:main} or Conjecture \ref{con:main}, it is perhaps helpful to pull on the thread of these motivations for the Property F conjecture a bit more so that we can compare and contrast.

On one hand, there is no ``unconditional" implication known between classification problems (2) and (4) above in either direction, except if we condition on properties in a way we have already mentioned, namely: if an anyon has Property F, then it is definitely not braiding universal, while if an anyon has dense braidings, then it is universal.
This is not ``unconditional" in the sense that as far as problem (2) is concerned, there is a third case that remains to be addressed: anyons with braidings that are neither dense nor have Property F.
Do they even exist?
If so, what are we to make of them?
Are they universal or not?
Maybe sometimes they are and sometimes they are not?
Conversely, if an anyon is braiding universal, does it necessarily have dense braid group representations?
These are interesting questions worth pursuing, but it could require quite a bit of effort to resolve each of them.

On the other hand, there is an ``unconditional" connection between (3) and (4) in at least one direction: (3) is simply the special case of (4) where the type of ``approximation" is chosen to be ``exact computation." 
So classification problem (3) might be understood as a warm-up to the version of problem (4) where the type of approximation is not ``exact", but is instead the kind of approximation relevant to topological quantum computing.  (For the sake of space, we refrain from precisely defining this type of approximation here; see the intro discussions of \cite{Kuperberg-hard} or \cite{me-hyperbolic}.)
The key technical issue that this perspective highlights is the following: even categories whose anyons all have property F (and thus are not braiding universal) can have $\shP$-hard invariants \cite{KirbyMelvin,me-zombies,me-coloring}.
Hence, more work needs to be done to properly understand the relationship between the $\BQP$-universality of anyon braidings in a given modular fusion category $\B$ and $\shP$-hardness of (exactly) computing $\tau_\B(M)$ on 3-manifolds $M$.
At the end of the day this is not so different from the situation between (2) and (4).

However, we conclude this discussion by noting that it is conceivable there exists a very tight connection between anyon classification problems (3) and (5) (while there is essentially no way to relate (2) and (5)).
Indeed, one might reasonably guess that $\shP$-hardness for the exact calculation of invariants implies that topological quantum computing with adaptive charge measurements is always sufficient to generate $\BQP$-universal topological gates.
This guess would be consistent with all known examples.

We plan to explore these matters in future work.

\subsection{Outline}
\label{ss:outline}
The next Section \ref{sec:proof} constitutes the remainder of this paper, and contains the proof of Theorem \ref{thm:main}.
Subsection \ref{ss:CaiChen} briefly reviews the definiton of $\#\CSP(\F)$, as well as Cai and Chen's dichotomy theorem for these problems.
Subsection \ref{ss:TVBW} contains the proof of part (a) of Theorem \ref{thm:main}, while Subsection \ref{ss:RT} contains the rather more-involved proof of part (b).

\section{Proof of Theorem \ref{thm:main}}
\label{sec:proof}

\subsection{Cai and Chen's dichotomy for weighted CSPs}
\label{ss:CaiChen}
Before proving either part of our main theorem, we review the definition of $\#\CSP(\F)$, following \cite{CaiChen-dichotomy}:
\begin{itemize}
\item We fix a finite set $D=\{1,\dots,d\}$ called the \emph{domain} (which, by an abuse of notation, we will suppress from the notation $\#\CSP(\F)$).
\item We fix a \emph{($\CC$-valued) weighted constraint family} $\F = \{f_1,\dots,f_h\}$, where each $f_i$ is a $\CC$-valued function $f_i: D^{r_i} \to \CC$ for some $r_i \ge 1$ called the \emph{arity} of $f_i$.  We assume all the values that the $f_i$ assume are encoded as algebraic numbers over $\QQ$.
\item An \emph{instance} $I$ of $\#\CSP(\F)$ consists of a tuple $\x = (x_1,\dots,x_n)$ of variables over $D$ (which will be suppressed in our notation) and a set $I$ of tuples $(f,i_1,\dots,i_r)$ in which $f$ is an $r$-ary function from $\F$ and $i_1,\dots,i_r \in \{1,\dots,n\}$ are indices of the variables in $\x$.
\item The \emph{output} of $\#\CSP(\F)$ on instance $I$ is the algebraic number $Z(I) \in \CC$ given by
\[Z(I) \defeq \sum_{x \in D^n} F_I(\x),\]
where
\[ F_I(\x) \defeq \prod_{(f,i_1,\dots,i_r) \in R} f(x_{i_1},\dots,x_{i_r}).\]
\end{itemize}

The main result of \cite{CaiChen-dichotomy} is

\begin{thm}[Thm.~1, \cite{CaiChen-dichotomy}]
Given any constraint set $\F$ as above, $\#\CSP(\F)$ is either computable in polynomial time or $\shP$-hard.
\label{thm:dichotomy}
\end{thm}

\subsection{Proof of Theorem \ref{thm:main}(a)}
\label{ss:TVBW}
Let $\C$ be a spherical fusion category over $\CC$.
To prove part (a) of Theorem \ref{thm:main}, it suffices---thanks to Theorem \ref{thm:dichotomy}---to build a domain $D_\C$ and weighted constraint set $\F_\C$ with the following property: there exists a polynomial time algorithm that converts a triangulated 3-manifold $M$ into an instance $I_M$ of $\#\CSP(\F_\C)$ such that
\[Z(I_M) = |M|_\C.\]
Readers already familiar with the state-sum formula for TVBW invariants will notice that the definition of $Z(I)$ is quite similar in spirit.
The goal of the present proof is simply to make this similarity precise. 

To this end, let us recall the state-sum formula for $|M|_\C$:
\[ |M|_\C = \D^{-2|V_M|}\sum_{\substack{ {L: E_M \to \Irr(\C)} \\ {F_L: F_M \to \N } \text{ consistent w/ $L$}}}  \frac{\prod_{e\in E_M} \dim(L(e))^2 \prod_{t\in T_M}|t^L|}{\prod_{f\in F_M} |f^L|}\]
where our notation is as follows:
\begin{itemize}
\item $V_M$ is the ordered list of vertices in the triangulation $M$ and $\D$ is the total quantum dimension of $\C$.
\item $E_M$ is the set of edges in the triangulation $M$ and $\Irr(\C)$ is set of simple objects in the given skeletalization of $\C$. 
\item $F_M$ is the set of faces in the triangulation $M$ and $\N$ is the set of labels of the trivalent Hom spaces \[ \mathrm{Hom}(k,i\otimes j) = \mathrm{span}\Bigg\{\scalebox{1}{
\begin{tikzpicture}[baseline = 0cm]
	\draw[ultra thick, postaction = {decorate, decoration= {markings, mark=at position .75 with {\arrow{<}}}}] (0,0) -- (-0.5,0.5)node[left]{$i$};
	\draw[ultra thick, postaction = {decorate, decoration= {markings, mark=at position .75 with {\arrow{<}}}}] (0,0) -- (0.5,0.5)node[right]{$j$};
	\draw[ultra thick, postaction = {decorate, decoration= {markings, mark=at position .8 with {\arrow{>}}}}] (0,0) -- (0,-0.5)node[left]{$k$};
	\node[draw=black, rectangle, fill=white, inner sep=2.5pt, minimum width = 20pt] at (0,0) {$\alpha$};
\end{tikzpicture}}\Bigg\}_{\alpha = 1, \dots , N_{ij}^k}\] and $|f^L|$ is the 3j+1k-symbol obtained by evaluating the face $f$ with a given labeling $L$ of the edges and faces of $M$ (See Figure \ref{fig:3j+1k}).
\item $T_M$ is the set of tetrahedra in the triangulation $M$, and $|t^L|$ is the 6j+4k-symbol obtained by evaluating the tetrahedron $t$ with a given labeling $L$ of the edges and faces of $M$, where we take into account whether the orientation of $t$ given by the orientation of $M$ matches the induced orientation given by the ordering of the vertices. This will be made more precise below.
\end{itemize}

\begin{figure}
	\begin{tikzpicture}
		\node at (-2.5,0.5) {$\phi(j_1,j_2,j_3;k)\quad =$};
		\draw[ultra thick, postaction = {decorate, decoration= {markings, mark=at position .5 with {\arrow{<}}}}] (0.25,0) -- (0.25,1.5);
		\draw[ultra thick, postaction = {decorate, decoration= {markings, mark=at position .5 with {\arrow{<}}}}] (-0.25,0) -- (-0.25,1.5);
		\draw[ultra thick] (0,1.5) arc (180:0:0.75);
		\draw[ultra thick, postaction = {decorate, decoration= {markings, mark=at position .5 with {\arrow{<}}}}] (1.5,1.5) -- (1.5,0);
		\draw[ultra thick] (1.5,0) arc (0:-180:0.75);
		\node[draw=black, fill=white, rectangle, minimum width = 1cm] at (0,0) {$k$};
		\node[draw=black, fill=white, rectangle, minimum width = 1cm] at (0,1.5) {$k$};
		\node[left] at (-0.25,0.75) {$j_1$};
		\node[right] at (0.25,0.75) {$j_2$};
		\node[right] at (1.5,0.75) {$j$};
	\end{tikzpicture}
\caption{A 3j+1k-symbol}
\label{fig:3j+1k}
\end{figure}
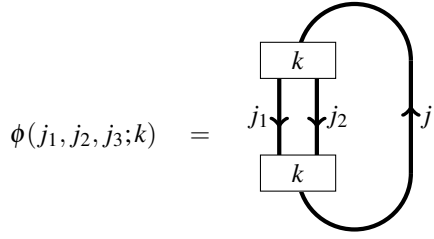


Since we assume $\C$ is not necessarily multiplicity-free, then instead of 6j-symbols, we will be using so-called 6j+4k-symbols $\begin{bmatrix} j_1 & j_2 & j_3 & k_{1,2} & k_{2,3} \\ j & j_{12} & j_{23} & k_{12,3} & k_{1,23} \end{bmatrix}^{\pm}$, which are defined by the contraction of a specific colored graph.

\begin{center}
\begin{tikzpicture}
	\node at (-5,0) (symbol) {$\begin{bmatrix} j_1 & j_2 & j_{12} & k_{1,2} & k_{2,3} \\ j_3 & j & j_{23} & k_{12,3} & k_{1,23} \end{bmatrix}^+ \quad \defeq$ };
        \draw[ultra thick, postaction={decorate, decoration={markings, mark=at position .5 with {\arrow{>}}}}, decoration={markings, mark=at position .5 with {\node[left] {$j_1$};}}] (-0.25,2) .. controls +(0,-1) and +(0,1) .. (-1.25,-0.75);
        \draw[ultra thick, postaction={decorate, decoration={markings, mark=at position .5 with {\arrow{>}}}}, decoration={markings, mark=at position .5 with {\node[above right] {$j_{23}$};}}] (0.25,2) .. controls +(0,-1) and +(0,1) .. (1,0.75);
        \draw[ultra thick, postaction={decorate, decoration={markings, mark=at position .5 with {\arrow{>}}}},decoration={markings, mark=at position .5 with {\node[below] {$j_2$};}}] (0.75,0.75) .. controls +(0,-1) and +(0,1) .. (-0.75,-0.75);
        \draw[ultra thick, postaction={decorate, decoration={markings, mark=at position .5 with {\arrow{>}}}}, decoration={markings, mark=at position .5 with {\node[right] {$j_3$};}}] (1.25,0.75) .. controls +(0,-1) and +(0,1) .. (0.25,-2);
        \draw[ultra thick, postaction={decorate, decoration={markings, mark=at position .5 with {\arrow{>}}}}, decoration={markings, mark=at position .5 with {\node[below left] {$j_{12}$};}}] (-1,-0.75) .. controls +(0,-1) and +(0,1) .. (-0.25,-2);
        \draw[ultra thick, postaction={decorate, decoration={markings, mark=at position .5 with {\arrow{<}}}}] (0,2) arc(180:0:1);
        \draw[ultra thick, postaction={decorate, decoration={markings, mark=at position .5 with {\arrow{<}}}}, decoration={markings, mark=at position .5 with {\node[right] {$j$};}}] (2,2) -- (2,-2);
        \draw[ultra thick, postaction={decorate, decoration={markings, mark=at position .5 with {\arrow{<}}}}] (2,-2) arc(0:-180:1);
        \node[draw=black, fill=white, rectangle, minimum width = 1cm] at (0,2) {$k_{1,23}$};
        \node[draw=black, fill=white, rectangle, minimum width = 1cm] at (0,-2) {$k_{12,3}$};
        \node[draw=black, fill=white, rectangle, minimum width = 1cm] at (1,0.75) {$k_{2,3}$};
        \node[draw=black, fill=white, rectangle, minimum width = 1cm] at (-1,-0.75) {$k_{1,2}$};
    \end{tikzpicture}
\end{center}
This defines the ``positive" 6j+4k-symbols. We also define a ``negative" version of the 6j+4k-symbols. We will call them negative 6j+4k-symbols since they correspond to negatively-oriented tetrahedra with respect to the standard orientation on $\mathbb{R}^3$:
\begin{center}
\begin{tikzpicture}
	\node at (-5,0) (symbol) {$\begin{bmatrix} j_1 & j_2 & j_{12} & k_{1,2} & k_{2,3} \\ j_3 & j & j_{23} & k_{12,3} & k_{1,23} \end{bmatrix}^-\quad \defeq$ };
        \draw[ultra thick, postaction={decorate, decoration={markings, mark=at position .5 with {\arrow{>}}}}, decoration={markings, mark=at position .5 with {\node[left] {$j_1$};}}] (-1.25,0.75) .. controls +(0,-1) and +(0,1) .. (-0.25,-2);
        \draw[ultra thick, postaction={decorate, decoration={markings, mark=at position .5 with {\arrow{>}}}}, decoration={markings, mark=at position .5 with {\node[below right] {$j_{23}$};}}] (1,-0.75) .. controls +(0,-1) and +(0,1) .. (0.25,-2);
        \draw[ultra thick, postaction={decorate, decoration={markings, mark=at position .5 with {\arrow{>}}}},decoration={markings, mark=at position .5 with {\node[below] {$j_2$};}}] (-0.75,0.75) .. controls +(0,-1) and +(0,1) .. (0.75,-0.75);
        \draw[ultra thick, postaction={decorate, decoration={markings, mark=at position .5 with {\arrow{>}}}}, decoration={markings, mark=at position .5 with {\node[right] {$j_3$};}}] (.25,2) .. controls +(0,-1) and +(0,1) .. (1.25,-0.75);
        \draw[ultra thick, postaction={decorate, decoration={markings, mark=at position .5 with {\arrow{>}}}}, decoration={markings, mark=at position .5 with {\node[above left] {$j_{12}$};}}] (-0.25,2) .. controls +(0,-1) and +(0,1) .. (-1,0.75);
        \draw[ultra thick, postaction={decorate, decoration={markings, mark=at position .5 with {\arrow{<}}}}] (0,2) arc(180:0:1);
        \draw[ultra thick, postaction={decorate, decoration={markings, mark=at position .5 with {\arrow{<}}}}, decoration={markings, mark=at position .5 with {\node[right] {$j$};}}] (2,2) -- (2,-2);
        \draw[ultra thick, postaction={decorate, decoration={markings, mark=at position .5 with {\arrow{<}}}}] (2,-2) arc(0:-180:1);
        \node[draw=black, fill=white, rectangle, minimum width = 1cm] at (0,2) {$k_{12,3}$};
        \node[draw=black, fill=white, rectangle, minimum width = 1cm] at (0,-2) {$k_{1,23}$};
        \node[draw=black, fill=white, rectangle, minimum width = 1cm] at (-1,0.75) {$k_{1,2}$};
        \node[draw=black, fill=white, rectangle, minimum width = 1cm] at (1,-0.75) {$k_{2,3}$};
    \end{tikzpicture}
\end{center}
Here, $j_1,j_2,j_3,j,j_{12},j_{23}$ are simple objects, $k_{1,2}\in \{0,\dots,N_{j_1 j_2}^{j_{12}}\}$, $k_{2,3}\in \{0,\dots, N_{j_2 j_3}^{j_{23}}\}$, $k_{12,3}\in \{0,\dots, N_{j_{12} j_3}^j\}$, and $k_{1,23}\in \{0,\dots ,N_{j_1,j_{23}}^j\}$.

We now have enough to identify our domain and weighted constraint set. Define \[D_\C \defeq \Irr(\C) \sqcup \N \sqcup \{*\}.\]
Now extend the 6j+4k symbols to be 10-ary functions on our domain $D_\C$ in the ``trivial" way:
\[\Delta^+(x_1, x_2, \dots ,x_{10}) \defeq \begin{dcases}
\begin{bmatrix} x_1 & x_2 & x_5 & x_7 & x_8 \\ x_3 & x_4 & x_6 & x_9 & x_{10} \end{bmatrix}^+ & \text{if } x_1 ,\dots ,x_6\in \Irr(\C)\text{ and }x_7,\dots ,x_{10} \in \N,\\
0 & \text{otherwise,}
\end{dcases} \] and \[ \Delta^-(x_1, x_2, \dots ,x_{10}) \defeq \begin{dcases}
\begin{bmatrix} x_1 & x_2 & x_5 & x_7 & x_8 \\ x_3 & x_4 & x_6 & x_9 & x_{10} \end{bmatrix}^- & \text{if } x_1 ,\dots ,x_6\in \Irr(\C)\text{ and }x_7,\dots ,x_{10} \in \N,\\
0 & \text{otherwise.}
\end{dcases} \] 
We similarly define 4-ary functions on our domain using the 3j+1k-symbols $\phi$ by taking 
\[\Phi^{-1}(x_1,x_2,x_3,x_4) \defeq \begin{cases}\phi(x_1,x_2,x_3;x_4)^{-1} & \text{if } x_1,x_2,x_3\in \Irr(\C) \text{ and } x_4\in \N,\\
0 & \text{otherwise.}\end{cases}\]
And we define 1-ary functions using the quantum dimensions of simple objects:
 \[ \mathrm{d}^2(x) \defeq \begin{cases}\dim(x)^2 & \text{if } x\in \Irr(\C), \\
 0 & \text{otherwise.}\end{cases}\]
Finally, we define a 1-ary function to encode the total quantum dimension of $\C$: \[\D^{-2} \defeq \D^{-2}(x) \defeq \begin{cases} \left (\sum_{j\in \Irr(\C)} \dim(j)^2 \right )^{-1} & \text{if } x=*, \\ 0 & \text{otherwise}.\end{cases}\]
Using these functions, we define our weighted constraint family \[\F_\C \defeq \{\Delta^\pm, \Phi^{-1},  \mathrm{d}^2, \D^{-2}\}.\]

Our next goal is to describe how to convert a triangulation $M$ of an oriented manifold into an instance $I_M$ of $\#\CSP(\F_\C)$. 

The data of $M$ is comprised of:
\begin{itemize}
\item An ordered list of vertices $\{v_1, \dots, v_a\}$.
\item A list of oriented edges $\{e_1(v_{1_1}, v_{2_1}), \dots , e_b(v_{b_1}, v_{b_2})\}$ where $e_i(v_{i_1}, v_{i_2})$ means that $e_i$ is an edge connecting $v_{i_1}$ to $v_{i_2}$. (Note that these orientations are chosen arbitrarily.)
\item A list of oriented faces $\{ f_1(e_{1_1}, e_{2_1}, e_{3_1}), \dots f_c(e_{c_1}, e_{c_2}, e_{c_3})\}$ where $f_i(e_{i_1},e_{i_2},e_{i_3})$ means that $f_i$ is a face whose boundary consists of the edges $e_{i_1}$, $e_{i_2}$, and $e_{i_3}$. (Note that the orientations of the faces need to be consistent with the edge orientations in any way.)
\item A list of tetrahedra $\{t_1, \dots,t_d\}$ where $t_i = t_i(f_{i_1}, \dots, f_{i_4})$ means that $t_i$ is a tetrahedron with faces given by $f_{i_1}$, \dots, $f_{i_4}$.
\item To encode the orientation of $M$, each tetrahedron $t_i$ is endowed with a sign $+$ or $-$ to indicate the local orientation inside that tetrahedron.\footnote{These signs must assemble to give a $\{\pm\}$-valued 0-cocycle on the dual cellulation.  This condition could be easily checked, but for our purposes it is simply part of the data structure of $M$, and so this condition can be assumed to be met as a promise.  This condition is not necessary for our proof (although it is necessary for the proof that $|M|_\C$ is an invariant of $M$).}
\end{itemize}

For a given triangulation $M$ as described, define a tuple
\[\x_M \defeq (x_1, \dots, x_a, y_1, \dots, y_b, z_1,\dots,z_c)\}\]
that has a variable for each vertex, edge and face in $M$.
We now describe how to build the desired instance $I_M$ of $\#\CSP(\F_\C)$.
It will be clear from the construction that the mapping $M \mapsto I_M$ can be done in polynomial time in the size of $M$.

First we put the functions $\D^{-2}(x_1), \dots , \D^{-2}(x_a)$ and $\mathrm{d}^2(y_1), \dots, \mathrm{d}^2(y_b)$ in $I_M$ for every vertex and edge of $M$.
For each face $f_j(e_{j_1}, e_{j_2}, e_{j_3})$, we include $\Phi^{-1}(y_{j_1},y_{j_2},y_{j_3},z_j)$.
Finally, for each tetrahedron $t_i$, we include either
\[\Delta^{+}(y_{j_1}, \dots ,y_{j_6}, z_{i_1}, \dots ,z_{i_4})),\]
or
\[\Delta^{-}(y_{j_1}, \dots ,y_{j_6}, z_{i_1}, \dots ,z_{i_4})),\]
where $t_i$ has faces $f_{i_1}(e_{j_1},e_{j_2},e_{j_5})$, $f_{i_2}(e_{j_5},e_{j_3},e_{j_4})$, $f_{i_3}(e_{j_3},e_{j_2},e_{j_6})$, and $f_{i_4}(e_{j_6}, e_{j_1}, e_{j_4})$.  To determine whether we should include $\Delta^+$ or $\Delta^-$ for $t_i$, we check if the orientation of $t_i$ given by the orientation of $M$ matches the induced orientation by the ordering of the vertices; if they match, then we use $\Delta^+$, and otherwise we use $\Delta^-$.

This $I_M$ defines an instance of $\#\CSP(\F_\C)$ that computes the Turaev-Viro invariant for $M$.
Indeed, plugging in the definitions of our constraint functions, we get
\[Z(I_M) = \sum_{\x\in D^{a+b+c}} \prod_{v = 1}^a \D^{-2}(x_v) \prod_{e=1}^b \mathrm{d}^2(y_e) \prod_{F_i=(E_{i_1},E_{i_2},E_{i_3})} \Phi^{-1}(y_{i_1},y_{i_2},y_{i_3},z_i)\prod_{T_i = (F_{i_1},\dots, F_{i_4})} \Delta^{\eta_i}(y_{j_1}, \dots ,y_{j_6}, z_{i_1}, \dots ,z_{i_4})) \]
where $\eta_i = +$ if the orientation of $T_i$ given by the orientation of $M$ matches the induced orientation by the ordering of the vertices, and $\eta_i = -$ otherwise.
\emph{A priori}, $Z(I_M)$ includes a sum over more types of labelings than the state-sum formula for $|M|_\C$.
However, because of how we have chosen to define the functions in the constraint family $\F_\C$, all of these additional terms in the sum vanish.
To see this, first note that when $x_1, \dots ,x_v \neq *$, the entire term is 0.  In particular, this means all non-zero terms have a common factor of the global quantum dimension to the $-a$ power, and hence we can pull it out as the normalizing factor.
Similarly, when the edges are not labeled by elements of the domain $D_\C$ that are not simple objects of $\C$, or the faces are not labeled with multiplicities, the terms are zero.
We have furthermore arranged so that when the labeling of the edges and faces is not admissible, the 6j+4k-symbol for that term vanishes.
Therefore, the only surviving terms in the sum $Z(I_M)$ are the $\x$ which define admissible labelings of the edges and faces of $M$.
It is then straightforward to see $Z(I_M) = |M|_\C$ recovers the Turaev-Viro invariant.
\qed

\subsection{Proof of Theorem \ref{thm:main}(b)}
\label{ss:RT}
Before proving part (b) of Theorem \ref{thm:main}, we establish a technical result about the graphical calculus in a spherical fusion category.
The result is likely well-known to experts, but does not appear in the literature anywhere that we are aware. We begin by reviewing what we need of closed trivalent graphs in $S^2$. 

For our purposes, a \emph{(closed) trivalent graph} $\Gamma$ in $\RR^2$ (or $S^2 = \RR^2\cup \{\infty\}$) has:
\begin{itemize}
\item A finite collection $V$ of vertices $v_1,v_2,\dots,v_n$, where $v_i\in \RR\times \{i\}$ 
\item A collection $E$ of directed edges $e_1,e_2, \dots, e_k$ in $S^2$ 
\end{itemize}
subject to the conditions that
\begin{itemize}
\item Each edge $e_i$ is either a loop disjoint from $V$ or an arc connecting two (not necessarily distinct) vertices, with an interior that is disjoint from all vertices in $V$.
\item If $v$ is a vertex, then there is an open disk neighborhood $D(v)$ so that $D(v) \cap \Gamma$ has three arcs (coming from intersections of $D(v)$ with E, not necessarily distinct) emanating from $v$ with one arc parallel to the vector $\langle0,1\rangle$, one arc parallel to the vector $\langle 1,1\rangle$, and one arc parallel to the vector $\langle -1,1\rangle$. 
\end{itemize}
Such a graph $\Gamma$ is closed in the sense that there are no vertices that are involved in precisely one half-edge.

We will also need to consider closed trivalent graphs which have crossings. We will call these \emph{crossed (closed) trivalent graphs}. Crossed trivalent graphs are trivalent graphs where they are allowed to have finitely many double points in the interior of its edges. These double points must indicate which segment of the edge crosses ``over" or ``under" the other.

Recall that computing the Reshetikhin-Turaev invariant $\tau_\B(M)$ of a 3-manifold $M$ presented by a surgery diagram involves a process where we interpret a coloring of the components of the diagram by simple objects of $\B$ as an endomorphism of the tensor unit $\oone$ of $\B$.
Since $\oone$ is itself simple, such a coloring gives rise to an endomorphism $\oone \to \oone$, which in turn can be identified with a complex number because $\End(\oone) = \CC$.
The invariant $\tau_\B(M)$ is then (roughly) the sum of all of these numbers over all choices of colorings of the surgery diagram of $M$.
For any single coloring, the complex number associated to it can generally be understood as the result of a sequence of tensor contractions on a tensor induced by that coloring.
Moreover, this sequence of tensor contractions can be represented diagramatically, using a small number of standard diagrammatic operations that are determined from the data of the modular fusion category $\B$.
We call these operations Circle Removal, Tadpole Trim, Bubble Pop, $F$-Move, and Vertex Spiral (aka ``bending" moves); see Figures \ref{fig:type0}-\ref{fig:vertexspiral}.
We also allow ourselves to reverse edge orientations.
To compute the complex number associated to a colored surgery link, one must identify a sequence of these operations that simplifies the diagram to the empty diagram.
The desired number will then be a product of numbers determined from the operations in the sequence and the given coloring.

Our proof of Theorem \ref{thm:main}(b) essentially revolves around two key observations.

First, a kind of uniformity: given a surgery description of $M$, there exists a \emph{single sequence} of diagrammatic operations that can be used to evaluate \emph{all} colorings of the surgery link to complex numbers.
This uniformity is, more-or-less, what will make it possible to encode $M \mapsto \tau_\B(M)$ as an instance $I_M$ of $\#\CSP(\F)$ for an appropriately chosen $\F$.

Second: such a uniform sequence of operations can be identified in polynomial time from the surgery description of $M$.
This will imply that the reduction $M \mapsto I_M$ can be performed in polynomial time. 
We make this point more precise now.

\begin{figure}
	\begin{tikzpicture}
		\draw[ultra thick, postaction={decorate, decoration={markings, mark=at position .5 with {\arrow{<}}}}] (3.5,0) circle (0.5);
		\begin{scope}[shift = {(1,1)}]
			\draw[ultra thick, postaction={decorate, decoration = {markings, mark=at position 0.5 with {\arrow{>}}}}] (0,0) .. controls +(-1,1) and +(1,1) .. (0,0);
			\draw[ultra thick,postaction={decorate, decoration = {markings, mark=at position 0.5 with {\arrow{>}}}}] (0,0) -- (0,-1);
			\draw[ultra thick, postaction={decorate, decoration = {markings, mark=at position 0.5 with {\arrow{>}}}}] (-1,-2) -- (0,-1);
			\draw[ultra thick,postaction={decorate, decoration = {markings, mark=at position 0.65 with {\arrow{>}}}}] (0,-1) -- (0.5,-1.5);
			\draw[ultra thick, postaction={decorate, decoration = {markings, mark=at position 0.75 with {\arrow{>}}}}] (0.5,-1.5) .. controls +(1,1) and +(0.25,0.25) .. (1,-2.5);
			\draw[ultra thick, postaction={decorate, decoration = {markings, mark=at position 0.5 with {\arrow{<}}}}] (0.5,-1.5) .. controls +(0,-0.5) and +(-0.5,0.5) .. (1,-2.5);
			\draw[ultra thick, postaction={decorate, decoration = {markings, mark=at position 0.5 with {\arrow{<}}}}] (1,-2.5) -- (1,-3);
			\draw[ultra thick, postaction={decorate, decoration = {markings, mark=at position 0.5 with {\arrow{>}}}}] (-1,-2) -- (-1,-3);
			\draw[ultra thick, postaction={decorate, decoration = {markings, mark=at position 0.5 with {\arrow{<}}}}] (-1,-2) .. controls +(-0.25,0.25) and +(0,-0.5) .. (-1.5,-0.5);
			\draw[ultra thick, postaction={decorate, decoration = {markings, mark=at position 0.5 with {\arrow{>}}}}] (-1,-3) arc (-180:0:1);
			\draw[ultra thick, postaction={decorate, decoration = {markings, mark=at position 0.5 with {\arrow{>}}}}] (-1.5,-0.5) .. controls +(-1,1) and +(1,1) .. (-1.5,-0.5);
		\end{scope}
	\end{tikzpicture}
\caption{A trivalent graph in $\mathbb{R}^2$ which exhibits all three types of edges (it looks like a person waving!)}
\label{fig:trivalentgraphex}
\end{figure}
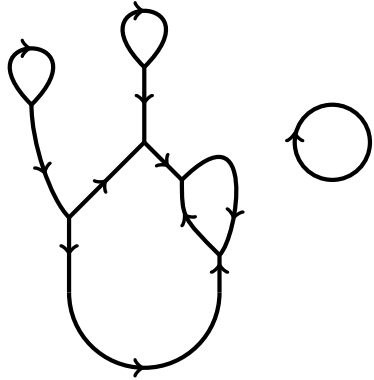

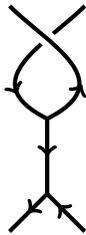
\begin{figure}
	\begin{tikzpicture}
		\draw[ultra thick, postaction = {decorate, decoration = {markings, mark = at position 0.75 with {\arrow{>}}}}] (0,2) .. controls +(-0.5,-0.5) and +(-1,0.5) .. (-0.5,0.5);
		\draw[white, line width = 6pt] (-1,2) .. controls +(0.5,-0.5) and +(1,0.5) .. (-0.45,0.45);
		\draw[ultra thick, postaction = {decorate, decoration = {markings, mark = at position 0.75 with {\arrow{<}}}}] (-1,2) .. controls +(0.5,-0.5) and +(1,0.5) .. (-0.5,0.5);
		\draw[ultra thick, postaction = {decorate, decoration = {markings, mark = at position 0.5 with {\arrow{>}}}}] (-0.5,0.5) -- (-0.5,-0.5);
		\draw[ultra thick, postaction = {decorate, decoration = {markings, mark = at position 0.5 with {\arrow{>}}}}] (-0.5,-0.5) -- (-1,-1);
		\draw[ultra thick, postaction = {decorate, decoration = {markings, mark = at position 0.5 with {\arrow{<}}}}] (-0.5,-0.5) -- (0,-1);
	\end{tikzpicture}
\caption{A portion of a crossed trivalent graph}
\label{fig:crossedtrivalentgraphex}
\end{figure}

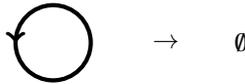
\begin{figure}
	\begin{tikzpicture}
		\draw[ultra thick,postaction={decorate, decoration = {markings, mark=at position 0.5 with {\arrow{>}}}}] (0,0) circle (0.5);
		\node at (1.5,0) {$\rightarrow$};
		\node at (2.5,0) {$\emptyset$};
	\end{tikzpicture}
\caption{Circle Removal}
\label{fig:type0}
\end{figure}

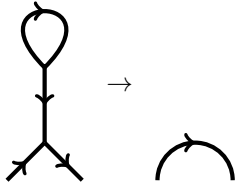
\begin{figure}
	\begin{tikzpicture}
		\draw[ultra thick,postaction={decorate, decoration = {markings, mark=at position 0.5 with {\arrow{>}}}}] (0,0) .. controls +(-1,1) and +(1,1) .. (0,0);
		\draw[ultra thick,postaction={decorate, decoration = {markings, mark=at position 0.5 with {\arrow{>}}}}] (0,0) -- (0,-1);
		\draw[ultra thick, postaction={decorate, decoration = {markings, mark=at position 0.5 with {\arrow{>}}}}] (-0.5,-1.5) -- (0,-1);
		\draw[ultra thick,postaction={decorate, decoration = {markings, mark=at position 0.65 with {\arrow{>}}}}] (0,-1) -- (0.5,-1.5);
		\node at (1,-0.25) {$\rightarrow$};
		\draw[ultra thick,postaction={decorate, decoration = {markings, mark=at position 0.5 with {\arrow{>}}}}] (1.5,-1.5) arc (180:0:0.5);
	\end{tikzpicture}
\caption{Tadpole Trim}
\label{fig:type1}
\end{figure}

\begin{figure}
	\begin{tikzpicture}
		\draw[ultra thick,postaction={decorate, decoration = {markings, mark=at position 0.5 with {\arrow{>}}}}] (0,1.5) -- (0,0.5);
		\draw[ultra thick,postaction={decorate, decoration = {markings, mark=at position 0.5 with {\arrow{<}}}}] (0,-0.5) .. controls +(-0.5,0.5) and +(-0.5,-0.5) .. (0,0.5);
		\draw[ultra thick,postaction={decorate, decoration = {markings, mark=at position 0.5 with {\arrow{<}}}}] (0,-0.5) .. controls +(0.5,0.5) and +(0.5,-0.5) .. (0,0.5); 
		\draw[ultra thick,postaction={decorate, decoration = {markings, mark=at position 0.5 with {\arrow{>}}}}] (0,-0.5) -- (0,-1.5);
		\node at (1,0) {$\rightarrow$};
		\draw[ultra thick,postaction={decorate, decoration = {markings, mark=at position 0.5 with {\arrow{>}}}}] (2,1.5) -- (2,-1.5);
	\end{tikzpicture}
\caption{Bubble Pop}
\label{fig:type2}
\end{figure}
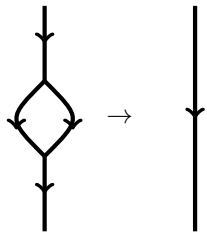

\begin{figure}
	\begin{tikzpicture}
		\draw[ultra thick, postaction={decorate, decoration={markings, mark=at position 0.5 with {\arrow{<}}}}] (-0.5,0) -- (-0.5,0.5);
		\draw[ultra thick, postaction={decorate, decoration={markings, mark=at position 0.5 with {\arrow{<}}}}] (-0.5,0.5) .. controls +(-0.25,0.25) and +(0,-0.25) .. (-1,1);
		\draw[ultra thick, postaction={decorate, decoration={markings, mark=at position 0.5 with {\arrow{<}}}}] (-1,1) -- (-1.5,1.5);
		\draw[ultra thick, postaction={decorate, decoration={markings, mark=at position 0.5 with {\arrow{<}}}}] (-0.5,0.5) -- (0.5,1.5);
		\draw[ultra thick, postaction={decorate, decoration={markings, mark=at position 0.5 with {\arrow{<}}}}] (-1,1) -- (-0.5,1.5);
		\node at (1,0.75) {$\leftrightarrow$};
		\begin{scope}[shift = {(2,0)}]
			\draw[ultra thick, postaction={decorate, decoration={markings, mark=at position 0.5 with {\arrow{<}}}}] (0.5,0) -- (0.5,0.5);
			\draw[ultra thick, postaction={decorate, decoration={markings, mark=at position 0.5 with {\arrow{<}}}}] (0.5,0.5) .. controls +(0.25,0.25) and +(0,-0.25) .. (1,1);
			\draw[ultra thick, postaction={decorate, decoration={markings, mark=at position 0.5 with {\arrow{<}}}}] (1,1) -- (1.5,1.5);
			\draw[ultra thick, postaction={decorate, decoration={markings, mark=at position 0.5 with {\arrow{<}}}}] (0.5,0.5) -- (-0.5,1.5);
			\draw[ultra thick, postaction={decorate, decoration={markings, mark=at position 0.5 with {\arrow{<}}}}] (1,1) -- (0.5,1.5);
		\end{scope}
		\begin{scope}[shift = {(7,1.5)}]
			\draw[ultra thick, postaction={decorate, decoration={markings, mark=at position 0.5 with {\arrow{>}}}}] (-0.5,0) -- (-0.5,-0.5);
			\draw[ultra thick, postaction={decorate, decoration={markings, mark=at position 0.5 with {\arrow{>}}}}] (-0.5,-0.5) .. controls +(-0.25,-0.25) and +(0,0.25) .. (-1,-1);
			\draw[ultra thick, postaction={decorate, decoration={markings, mark=at position 0.5 with {\arrow{>}}}}] (-1,-1) -- (-1.5,-1.5);
			\draw[ultra thick, postaction={decorate, decoration={markings, mark=at position 0.5 with {\arrow{>}}}}] (-0.5,-0.5) -- (0.5,-1.5);
			\draw[ultra thick, postaction={decorate, decoration={markings, mark=at position 0.5 with {\arrow{>}}}}] (-1,-1) -- (-0.5,-1.5);
			\node at (1,-0.75) {$\leftrightarrow$};
			\begin{scope}[shift = {(2,0)}]
				\draw[ultra thick, postaction={decorate, decoration={markings, mark=at position 0.5 with {\arrow{>}}}}] (0.5,0) -- (0.5,-0.5);
				\draw[ultra thick, postaction={decorate, decoration={markings, mark=at position 0.5 with {\arrow{>}}}}] (0.5,-0.5) .. controls +(0.25,-0.25) and +(0,0.25) .. (1,-1);
				\draw[ultra thick, postaction={decorate, decoration={markings, mark=at position 0.5 with {\arrow{>}}}}] (1,-1) -- (1.5,-1.5);
				\draw[ultra thick, postaction={decorate, decoration={markings, mark=at position 0.5 with {\arrow{>}}}}] (0.5,-0.5) -- (-0.5,-1.5);
				\draw[ultra thick, postaction={decorate, decoration={markings, mark=at position 0.5 with {\arrow{>}}}}] (1,-1) -- (0.5,-1.5);
			\end{scope}
		\end{scope}
	\end{tikzpicture}
\caption{$F$-Move}
\label{fig:typeF}
\end{figure}

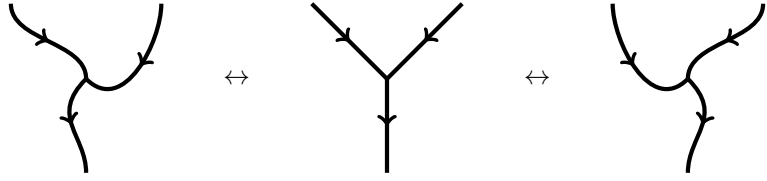
\begin{figure}
	\begin{tikzpicture}
		\draw[ultra thick, postaction={decorate, decoration={markings, mark=at position 0.5 with {\arrow{>}}}}] (0,0) -- (0,-1.25);
		\draw[ultra thick, postaction={decorate, decoration={markings, mark=at position 0.5 with {\arrow{>}}}}] (1,1) -- (0,0);
		\draw[ultra thick, postaction={decorate, decoration={markings, mark=at position 0.5 with {\arrow{>}}}}] (-1,1) -- (0,0);
		\node at (2,0) {$\leftrightarrow$};
		\begin{scope}[shift = {(4,0)}]
		\draw[ultra thick, postaction={decorate, decoration={markings, mark=at position 0.5 with {\arrow{>}}}}] (0,0) .. controls +(0.5,-0.5) and +(0,0.5) .. (0,-1.25);
		\draw[ultra thick, postaction={decorate, decoration={markings, mark=at position 0.5 with {\arrow{>}}}}] (1,1) .. controls +(0,-0.5) and +(0,0.5) .. (0,0);
		\draw[ultra thick, postaction={decorate, decoration={markings, mark=at position 0.5 with {\arrow{>}}}}] (-1,1) .. controls +(0,-0.5) and +(-0.5,-0.5) .. (0,0);
		\end{scope}
		\node at (-2,0) {$\leftrightarrow$};
		\begin{scope}[shift = {(-4,0)}]
		\draw[ultra thick, postaction={decorate, decoration={markings, mark=at position 0.5 with {\arrow{>}}}}] (0,0) .. controls +(-0.5,-0.5) and +(0,0.5) .. (0,-1.25);
		\draw[ultra thick, postaction={decorate, decoration={markings, mark=at position 0.5 with {\arrow{>}}}}] (1,1) .. controls +(0,-0.5) and +(0.5,-0.5) .. (0,0);
		\draw[ultra thick, postaction={decorate, decoration={markings, mark=at position 0.5 with {\arrow{>}}}}] (-1,1) .. controls +(0,-0.5) and +(0,0.5) .. (0,0);
		\end{scope}
	\end{tikzpicture}
\caption{Vertex Spiral (aka ``bending")}
\label{fig:vertexspiral}
\end{figure}

\begin{lem}\label{lem:trivalentgraphalg}
If $\Gamma \subset S^2$ is a closed trivalent graph embedded in $S^2$, then there is a polynomial time algorithm (in the size of the encoding of $\Gamma$) to construct a sequence of embedded graphs $\Gamma_0, \Gamma_1,\dots, \Gamma_l$ where $\Gamma_0 = \Gamma$, $\Gamma_l = \emptyset$ such that each $\Gamma_{i+1}$ is related to $\Gamma_i$ by one of the diagrammatic operations in Figures \ref{fig:type0}-\ref{fig:vertexspiral} or edge orientation reversals.
\end{lem}

\begin{proof}
A simple greedy algorithm suffices.  We sketch the idea.

Begin by greedily choosing a complementary region $R$ of $\Gamma$, \emph{i.e.}\ a connected component $R$ of $S^2 \setminus \Gamma$.
Note that we may identify the boundary edges and vertices of $R$ in polynomial time.
Suppose $R$ has $k$ unique edges and $l$ unique vertices on its boundary.
Now simplify and update $\Gamma$ according to the following cases.
\begin{enumerate}
\item If $(k,l)=(0,0)$, then $\Gamma=\emptyset$, and so we terminate.
\item If $(k,l)=(1,0)$, then the boundary of $R$ is a circle in $\Gamma$, which we remove as in Figure \ref{fig:type0}. 
\item If $(k,l) = (1,1)$, then the boundary of $R$ is part of a tadpole, which we trim as in Figure \ref{fig:type1}, but possibly only after first applying an appropriate set of vertex spirals as in Figure \ref{fig:vertexspiral} and edge orientation reversals. 
\item If $(k,l) = (2,2)$, then the boundary of $R$ is part of a bubble, which we pop as in Figure \ref{fig:type2}, but possibly only after first applying an appropriate set of vertex spirals and edge orientation reversals.
\item Otherwise, $k=l>2$.
Greedily pick an edge $e$ on the boundary of $R$.
After perhaps first applying up to two vertex spirals and 5 edge orientation reversals, we can arrange so that around $e$, $\Gamma$ looks like one of the four diagrams in Figure \ref{fig:typeF}, with $e$ the edge in the middle.
Apply the available $F$-move around $e$.
The complementary regions of the resulting graph are naturally in bijection with the regions of the previous graph (see Figure \ref{fig:algorithmexample} for an example).
Let $R'$ be the region of the new graph associated with $R$.
If $R'$ has $k=l>2$ edges on its boundary, then repeat what we just did, but with $R'$ and the new graph, instead of $R$;
otherwise, $R'$ has $k=l=2$ edges, and we pop the bubble as in case (3).
\end{enumerate}
Repeat this process of greedily picking a region $R$ and proceeding as in the above cases.
Each step of identifying an $R$ and carrying through the appropriate case takes polynomial time, and, moreover, reduces the number of complementary regions of $\Gamma$ by 1.
Since there are at most a polynomial number of complementary regions to begin with, the entire procedure takes place in polynomial time.
\end{proof}

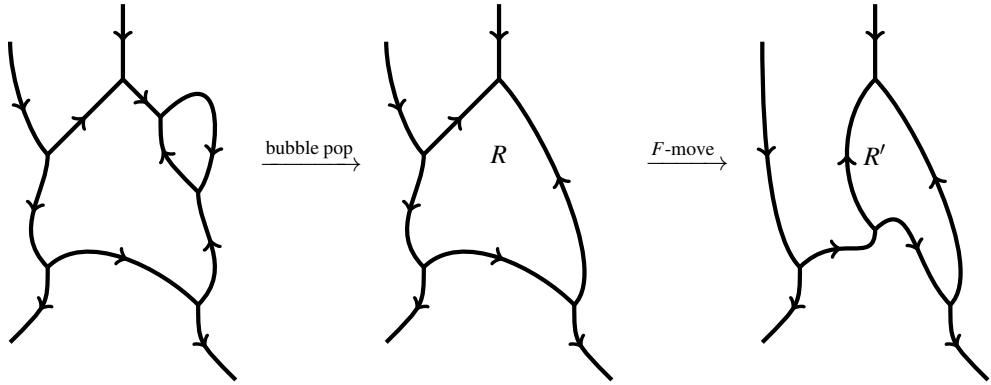
\begin{figure}
\begin{tikzpicture}
		\begin{scope}[shift = {(0,0)}]
			\draw[ultra thick,postaction={decorate, decoration = {markings, mark=at position 0.5 with {\arrow{>}}}}] (0,0) -- (0,-1);
			\draw[ultra thick, postaction={decorate, decoration = {markings, mark=at position 0.5 with {\arrow{>}}}}] (-1,-2) -- (0,-1);
			\draw[ultra thick,postaction={decorate, decoration = {markings, mark=at position 0.65 with {\arrow{>}}}}] (0,-1) -- (0.5,-1.5);
			\draw[ultra thick, postaction={decorate, decoration = {markings, mark=at position 0.75 with {\arrow{>}}}}] (0.5,-1.5) .. controls +(1,1) and +(0.25,0.25) .. (1,-2.5);
			\draw[ultra thick, postaction={decorate, decoration = {markings, mark=at position 0.5 with {\arrow{<}}}}] (0.5,-1.5) .. controls +(0,-0.5) and +(-0.5,0.5) .. (1,-2.5);
			\draw[ultra thick, postaction={decorate, decoration = {markings, mark=at position 0.5 with {\arrow{<}}}}] (1,-2.5) .. controls +(0,-0.5) and +(0.5,0.5) .. (1,-4);
			\draw[ultra thick, postaction={decorate, decoration = {markings, mark=at position 0.5 with {\arrow{>}}}}] (-1,-2) .. controls +(0,-0.5) and +(-0.5,0.5) .. (-1,-3.5);
			\draw[ultra thick, postaction={decorate, decoration = {markings, mark=at position 0.5 with {\arrow{<}}}}] (-1,-2) .. controls +(-0.25,0.25) and +(0,-0.5) .. (-1.5,-0.5);
			\draw[ultra thick, postaction={decorate, decoration = {markings, mark=at position 0.5 with {\arrow{>}}}}] (-1,-3.5) .. controls +(0.5,0.5) and +(-0.5,0.5) .. (1,-4);
			\draw[ultra thick, postaction={decorate, decoration = {markings, mark=at position 0.5 with {\arrow{>}}}}] (-1,-3.5) .. controls +(0,-0.5) and +(0.5,0.5) .. (-1.5,-4.5);
			\draw[ultra thick, postaction={decorate, decoration = {markings, mark=at position 0.5 with {\arrow{>}}}}] (1,-4) .. controls +(0,-0.5) and +(-0.5,0.5).. (1.5,-5);
			\node at (2.5,-2) {$\xrightarrow{\text{bubble pop}}$};
		\end{scope}
		\begin{scope}[shift = {(5,0)}]
			\draw[ultra thick,postaction={decorate, decoration = {markings, mark=at position 0.5 with {\arrow{>}}}}] (0,0) -- (0,-1);
			\draw[ultra thick, postaction={decorate, decoration = {markings, mark=at position 0.5 with {\arrow{>}}}}] (-1,-2) -- (0,-1);
			\draw[ultra thick, postaction={decorate, decoration = {markings, mark=at position 0.5 with {\arrow{<}}}}] (0,-1) .. controls +(0.5,-0.5) and +(0.5,0.5) .. (1,-4);
			\draw[ultra thick, postaction={decorate, decoration = {markings, mark=at position 0.5 with {\arrow{>}}}}] (-1,-2) .. controls +(0,-0.5) and +(-0.5,0.5) .. (-1,-3.5);
			\draw[ultra thick, postaction={decorate, decoration = {markings, mark=at position 0.5 with {\arrow{<}}}}] (-1,-2) .. controls +(-0.25,0.25) and +(0,-0.5) .. (-1.5,-0.5);
			\draw[ultra thick, postaction={decorate, decoration = {markings, mark=at position 0.5 with {\arrow{>}}}}] (-1,-3.5) .. controls +(0.5,0.5) and +(-0.5,0.5) .. (1,-4);
			\draw[ultra thick, postaction={decorate, decoration = {markings, mark=at position 0.5 with {\arrow{>}}}}] (-1,-3.5) .. controls +(0,-0.5) and +(0.5,0.5) .. (-1.5,-4.5);
			\draw[ultra thick, postaction={decorate, decoration = {markings, mark=at position 0.5 with {\arrow{>}}}}] (1,-4) .. controls +(0,-0.5) and +(-0.5,0.5).. (1.5,-5);
			\node at (2.5,-2) {$\xrightarrow{\text{$F$-move }}$};
			\node at (0,-2) {$R$};
		\end{scope}
		\begin{scope}[shift = {(10,0)}]
			\draw[ultra thick,postaction={decorate, decoration = {markings, mark=at position 0.5 with {\arrow{>}}}}] (0,0) -- (0,-1);
			\draw[ultra thick, postaction={decorate, decoration = {markings, mark=at position 0.5 with {\arrow{>}}}}] (0,-3) .. controls +(-0.5,0.5) and +(-0.5,-0.5) .. (0,-1);
			\draw[ultra thick, postaction={decorate, decoration = {markings, mark=at position 0.5 with {\arrow{<}}}}] (0,-1) .. controls +(0.5,-0.5) and +(0.5,0.5) .. (1,-4);
			\draw[ultra thick, postaction={decorate, decoration = {markings, mark=at position 0.5 with {\arrow{>}}}}] (-1.5,-0.5) .. controls +(0,-0.5) and +(-0.5,0.5) .. (-1,-3.5);
			\draw[ultra thick, postaction={decorate, decoration = {markings, mark=at position 0.5 with {\arrow{>}}}}] (-1,-3.5) .. controls +(0.5,0.5) and +(0,-0.5) .. (0,-3);
			\draw[ultra thick, postaction={decorate, decoration = {markings, mark=at position 0.5 with {\arrow{>}}}}] (-1,-3.5) .. controls +(0,-0.5) and +(0.5,0.5) .. (-1.5,-4.5);
			\draw[ultra thick, postaction={decorate, decoration = {markings, mark=at position 0.5 with {\arrow{>}}}}] (0,-3) .. controls +(0.5,0.5) and +(-0.5,0.5) .. (1,-4);
			\draw[ultra thick, postaction={decorate, decoration = {markings, mark=at position 0.5 with {\arrow{>}}}}] (1,-4) .. controls +(0,-0.5) and +(-0.5,0.5).. (1.5,-5);
			\node at (0,-2) {$R'$};
		\end{scope}

	\end{tikzpicture}
\caption{An example of a portion of the algorithm}
\label{fig:algorithmexample}
\end{figure}

\begin{proof}[Proof of Theorem \ref{thm:main}(b)]
Let $\B$ be a modular fusion category over $\mathbb{C}$. As in part (a), it suffices to build a domain $D_\B$ and weighted constraint set $\F_\B$ for which there is a polynomial time algorithm to encode a surgery diagram for a 3-manifold $M$ into an instance $I_M$ of $\# \CSP(\F_\B)$ such that \[Z(I_M) = \tau_\B(M)\]

Let us recall the formula for $\tau_\B(M)$: \[ \tau_\B(M) = p_-^{\frac{\sigma(L)-m-1}{2}}p_+^{\frac{-\sigma(L)-m-1}{2}} \sum_{\mathrm{col}:\{1,\dots,m\}\to \Irr(\B)}\left (\prod_{j=1}^m \dim(\mathrm{col}(j))\right )|L^\mathrm{col}|\] where our notation is as follows:
\begin{itemize}
	\item $L$ is the surgery link diagram that defines the 3-manifold $M$.  $L$ consists of $m$ components (labeled $1,2,\dots,m$) and, for convenience is endowed with the blackboard framing.\footnote{To justify this convenience, simply apply a Reidemeister move of type 1 to each the components of $L$ so that the blackboard framing agrees with the desired integral surgery coefficients.  This can be done in polynomial time because we encode the surgery coefficents in unary.}  We will also assume, for convenience, that $L$ is in ``standard plat position." This means that all of the local minima of the diagram (which, recall, is a picture in the $xy$-plane) occur below all crossings, all local maxima of the diagram occur above all crossings, and the sets of cups and caps have no ``nesting".\footnote{This convenience can be justified by the fact that any link diagram can be put in standard plat position in polynomial time by simply applying a sequence of Reidemeister 2 moves.}  This means $L$ is entirely determined by a braid word.  See Figure \ref{fig:plat}.
	\item $\sigma(L)$ is the signature of $L$: the number of positive eigenvalues of the linking matrix minus the number of negative eigenvalues.  (This can be computed in polynomial time from $L$.)
	\item $p_\pm = \sum_{i\in\Irr(\B)} \theta_i^{\pm} (\dim(i))^2$ are the Gauss sums of $\B$.
	\item $|L^{\mathrm{col}}|$ is the evaluation of the colored ribbon graph defined by coloring the components of $L$ by $\mathrm{col}$.
\end{itemize}

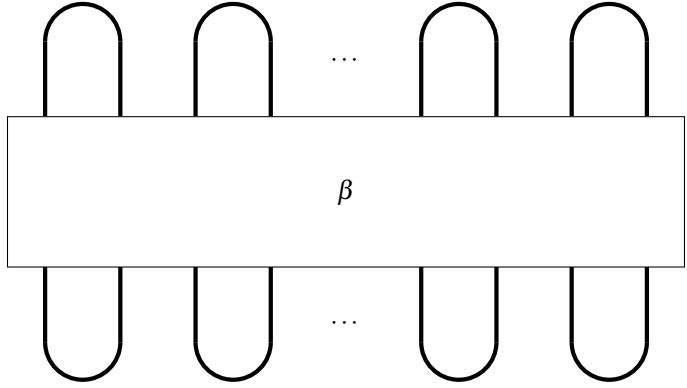
\begin{figure}
\label{fig:plat}
\begin{tikzpicture}
	\foreach \x in {-4,-3,-2,-1,1,2,3,4} {
		\draw[ultra thick] (\x,2) -- (\x,-2);
	}
	\foreach \x in {-4,-2,1,3} {
		\draw[ultra thick] (\x,-2) arc (-180:0:0.5);
		\draw[ultra thick] (\x,2) arc (180:0:0.5);
	}
	\node[draw = black, fill = white, rectangle, minimum width = 9cm, minimum height = 2cm] at (0,0) {$\beta$};
	\node at (0,1.75) {$\cdots$};
	\node at (0,-1.75) {$\cdots$};
\end{tikzpicture}
\caption{A blackboard-framed link $L$ in standard plat position is determined by a braid word $\beta \in B_{2k}$.}
\end{figure}

We now define \[D_\B \defeq \Irr(\B)\sqcup \N\sqcup \{*\}\] where $\N$ is the set of labels of the trivalent Hom space as in the proof of Theorem \ref{thm:main}(a). 
As hinted at in Lemma \ref{lem:trivalentgraphalg}, we need $\CC$-valued constraint functions on the domain $D_\B$ that implement bubble pops, tadpole removals, $F$-moves, \emph{etc}.
We define them as follows.

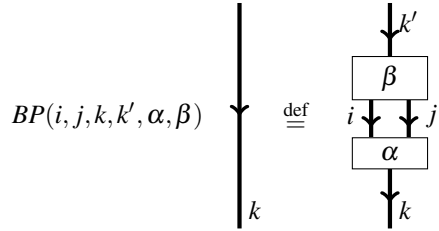
\begin{center}
\begin{tikzpicture}
	\node at (-3.75,0) {$BP(i,j,k,k',\alpha,\beta)$};
	\draw[ultra thick, postaction = {decorate, decoration = {markings, mark = at position 0.5 with {\arrow{>}}}}] (-2,1.5) -- (-2,-1.5);
	\node[right] at (-2,-1.25) {$k$};
	\node at (-1.25,0) {$\defeq$};
	\draw[ultra thick, postaction = {decorate, decoration = {markings, mark = at position 0.5 with {\arrow{>}}}}] (0,1.5) -- (0,0.5);
	\draw[ultra thick, postaction = {decorate, decoration = {markings, mark = at position 0.65 with {\arrow{>}}}}] (0,-0.5) -- (0,-1.5);
	\draw[ultra thick, postaction = {decorate, decoration = {markings, mark = at position 0.65 with {\arrow{>}}}}] (-0.25,0.5) -- (-0.25, -0.5);
	\draw[ultra thick, postaction = {decorate, decoration = {markings, mark = at position 0.65 with {\arrow{>}}}}] (0.25,0.5) -- (0.25, -0.5);
	\node[draw = black, rectangle, fill = white, minimum width = 1cm] at (0,0.5) {$\beta$};
	\node[draw=black, rectangle, fill=white, minimum width = 1cm] at (0,-0.5) {$\alpha$};
	\node[right] at (0,-1.25) {$k$};
	\node[right] at (0,1.25) {$k'$};
	\node[left] at (-0.35,0) {$i$};
	\node[right] at (0.35,0) {$j$};
\end{tikzpicture}
\end{center}
where $i,j,k,k'\in \Irr(\B)$, $\alpha\in \{1,\dots, N_{ij}^k\}$, and $\beta\in \{1,\dots, N_{ij}^{k'}\}$. These implement the bubble pop.

\begin{center}
\begin{tikzpicture}
	\node at (-3.5,0) {$TT(i,j,k,k',\alpha,\beta)$};
	\draw[ultra thick, postaction = {decorate, decoration = {markings, mark = at position 0.5 with {\arrow{>}}}}] (-2,-0.5) .. controls +(0,1) and +(0,1) .. (-1,-0.5);
	\node[right] at (-1,-0.5) {$k$};
	\node at (0,0) {$\defeq$};
	\begin{scope}[shift = {(1.5,0)}]
	\draw[ultra thick, postaction = {decorate, decoration = {markings, mark = at position 0.65 with {\arrow{>}}}}] (0,0.5) -- (0,-0.5);
	\draw[ultra thick, postaction = {decorate, decoration = {markings, mark = at position 0.5 with {\arrow{>}}}}] (-0.35,0.5) .. controls +(-0.25,1) and +(0.25,1) .. (0.35, 0.5);
	\draw[ultra thick, postaction = {decorate, decoration = {markings, mark = at position 0.65 with {\arrow{<}}}}] (-0.25,-0.5) -- (-0.25,-1.5);
		\draw[ultra thick, postaction = {decorate, decoration = {markings, mark = at position 0.65 with {\arrow{>}}}}] (0.25,-0.5) -- (0.25,-1.5);
	\node[draw = black, rectangle, fill = white, minimum width = 1cm] at (0,0.5) {$\beta$};
	\node[draw=black, rectangle, fill=white, minimum width = 1cm] at (0,-0.5) {$\alpha$};
	\node[right] at (0.35,-1.25) {$k$};
	\node[left] at (-0.35,-1.25) {$k'$};
	\node[left] at (-0.25,1.25) {$i$};
	\node[right] at (0,0) {$j$};
	\end{scope}
\end{tikzpicture}
\end{center}
where $i,j,k,k'\in \Irr(\B)$, $\alpha\in\{1,\dots,N_{k'^* k}^j\}$, $\beta\in \{1,\dots, N_{i^*i}^j\}$. These implement tadpole trims. 

We define functions $VS^{\mathrm{left}}$ by requiring
\begin{center}
\begin{tikzpicture}
	\draw[ultra thick, postaction = {decorate, decoration = {markings, mark = at position 0.65 with {\arrow{>}}}}] (-0.25,0) -- (-0.25,-1);
	\draw[ultra thick, postaction = {decorate, decoration = {markings, mark = at position 0.5 with {\arrow{>}}}}] (0,1) -- (0,0);
	\draw[ultra thick] (0.25,0) arc [start angle=-180, delta angle=180, x radius=0.25cm, y radius = 0.75cm];
	\draw[ultra thick, postaction = {decorate, decoration = {markings, mark = at position 0.65 with {\arrow{>}}}}] (0.75,0) -- (0.75,1);
	\node[draw = black, rectangle, fill = white, minimum width = 1cm] at (0,0) {$\alpha$};
	\node[left] at (-0.25,-0.75) {$i$};
	\node[right] at (0.75,0.75) {$j$};
	\node[left] at (0,0.75) {$k$};
	\node at (1.5,0) {$=$};
	\begin{scope}[shift = {(7,0)}]
		\node at (-3,0) {$\displaystyle\sum_{\substack{\beta\in \{1,\dots ,N_{kj^*}^i\}}} VS^{left}(i,j,k,\alpha,\beta)$};
		\draw[ultra thick, postaction = {decorate, decoration = {markings, mark = at position 0.65 with {\arrow{>}}}}] (0,0) -- (0,-1);
	\draw[ultra thick, postaction = {decorate, decoration = {markings, mark = at position 0.5 with {\arrow{>}}}}] (-0.25,1) -- (-0.25,0);
	\draw[ultra thick, postaction = {decorate, decoration = {markings, mark = at position 0.65 with {\arrow{>}}}}] (0.25,0) -- (0.25,1);
	\node[draw = black, rectangle, fill = white, minimum width = 1cm] at (0,0) {$\beta$};
	\node[left] at (0,-0.75) {$i$};
	\node[right] at (0.25,0.75) {$j$};
	\node[left] at (-0.25,0.75) {$k$};
	\end{scope}
\end{tikzpicture}
\end{center}
where $i,j,k\in \Irr(\B)$ and $\alpha\in \{1,\dots ,N_{ij}^k\}$. These implement the ``left" vertex spin.
We similarly define $VS^{\mathrm{right}}$ coefficients to implement the right vertex spin.

We need the $F$-matrices, which have coefficients $F^+$ that satisfy the equation
\begin{center}
\begin{tikzpicture}
	\draw[ultra thick, postaction = {decorate, decoration = {markings, mark = at position 0.5 with {\arrow{>}}}}] (-2,2) .. controls +(0.5,-0.5) and +(0,0.5) .. (-1.25,1);
	\draw[ultra thick, postaction = {decorate, decoration = {markings, mark = at position 0.5 with {\arrow{>}}}}] (-1,1) .. controls +(0,-0.75) and +(0,0.75) .. (-0.25,0);
	\draw[ultra thick, postaction = {decorate, decoration = {markings, mark = at position 0.65 with {\arrow{>}}}}] (0,0) -- (0,-1);
	\draw[ultra thick, postaction = {decorate, decoration = {markings, mark = at position 0.5 with {\arrow{<}}}}] (0.25,0) .. controls +(0,0.75) and +(-0.5,-0.75) .. (2,2);
	\draw[ultra thick, postaction = {decorate, decoration = {markings, mark = at position 0.5 with {\arrow{<}}}}] (-0.75,1) .. controls +(0,0.5) and +(-0.5,-0.5) .. (0,2);
	\node[draw=black, rectangle, fill=white, minimum width = 1cm] at (-1,1) {$\alpha$};
	\node[draw=black, rectangle, fill=white, minimum width = 1cm] at (0,0) {$\beta$};
	\node[left] at (-2,1.75) {$a$};
	\node[left] at (-0.5,1.75) {$b$};
	\node[left] at (1.5,1.75) {$c$};
	\node[left] at (-0.75,0.4) {$m$};
	\node[left] at (0,-0.75) {$d$};
	\node at (5,0) {$=\displaystyle\sum_{\substack{m\in \Irr(\B) \\ \delta\in\{1,\dots,N_{bc}^m\} \\ \gamma\in\{1,\dots, N_{am}^d\}}}F^+(a,b,c,d,m,n,\alpha,\beta,\delta,\gamma)$};
	\begin{scope}[shift = {(10,0)}]
	\draw[ultra thick, postaction = {decorate, decoration = {markings, mark = at position 0.5 with {\arrow{>}}}}] (-2,2) .. controls +(0.5,-0.75) and +(0,0.75) .. (-0.25,0);
	\draw[ultra thick, postaction = {decorate, decoration = {markings, mark = at position 0.5 with {\arrow{>}}}}] (2,2) .. controls +(-0.5,-0.5) and +(0,0.5) .. (1.25,1);
	\draw[ultra thick, postaction = {decorate, decoration = {markings, mark = at position 0.5 with {\arrow{>}}}}] (1,1) .. controls +(0,-0.75) and +(0,0.75) .. (0.25,0);
	\draw[ultra thick, postaction = {decorate, decoration = {markings, mark = at position 0.65 with {\arrow{>}}}}] (0,0) -- (0,-1);
	\draw[ultra thick, postaction = {decorate, decoration = {markings, mark = at position 0.5 with {\arrow{<}}}}] (0.75,1) .. controls +(0,0.5) and +(0.5,-0.5) .. (0,2);
	\node[draw=black, rectangle, fill=white, minimum width = 1cm] at (1,1) {$\delta$};
	\node[draw=black, rectangle, fill=white, minimum width = 1cm] at (0,0) {$\gamma$};
	\node[left] at (-2,1.75) {$a$};
	\node[left] at (0,1.75) {$b$};
	\node[left] at (1.5,1.75) {$c$};
	\node[right] at (0.75,0.4) {$m$};
	\node[left] at (0,-0.75) {$d$};
	\end{scope}
\end{tikzpicture}
\end{center}
where $a,b,c,d,m\in \Irr(\B)$, $\alpha\in \{1,\dots, N_{ab}^m\}$, and $\beta\in \{1,\dots, N_{mc}^d\}$. The inverse $F$-matrix has coefficients that we denote by $F^-$. We also need to include the matrix coefficients implementing the upside-down version of this picture. We call these $G^\pm$, respectively.

We then extend the above to 6-ary and 10-ary functions on our domain:

\[F^\pm(x_1,x_2,\dots ,x_{10}) \defeq \begin{cases}
F^\pm(x_1,x_2,\dots, x_{10}) & \text{ if }x_1,\dots ,x_6\in \Irr(\B)\text{ and }x_7,\dots ,x_{10}\in \N\\ 0 & \text{ otherwise}
\end{cases}\] 

\[G^\pm(x_1,x_2,\dots ,x_{10}) \defeq \begin{cases}
G^\pm(x_1,x_2,\dots, x_{10}) & \text{ if }x_1,\dots ,x_6\in \Irr(\B)\text{ and }x_7,\dots ,x_{10}\in \N\\ 0 & \text{ otherwise}
\end{cases}\] 

\[BP(x_1,x_2,\dots ,x_6) \defeq \begin{cases}
BP(x_1,x_2,\dots ,x_6) & \text{ if }x_1,x_2,x_3,x_4\in \Irr(\B) \text{ and }x_5,x_6\in \N \\ 0 & \text{ otherwise}
\end{cases}\] 

\[TT(x_1,x_2,\dots ,x_6) \defeq \begin{cases}
TT(x_1,x_2,\dots ,x_6) & \text{ if }x_1,x_2,x_3,x_4\in \Irr(\B) \text{ and }x_5,x_6\in \N \\ 0 & \text{ otherwise}
\end{cases}\] 

\[VS^*(x_1,x_2,x_3,x_4,x_5) \defeq \begin{cases}
VS^*(x_1,x_2,x_3,x_4,x_5) & \text{ if } x_1,x_2,x_3\in \Irr(\B) \text{ and }x_4,x_5 \in \N \\ 0 & \text{ otherwise}
\end{cases}\] for $*\in \{\mathrm{left},\mathrm{right}\}$.

We also need to implement braidings, which are described diagrammatically via $R$-moves. Recall the definition of the $R$-symbols: for $i,j,k\in \Irr(\B)$ and $\alpha\in \{1,\dots ,N_{ji}^k\}$, they satisfy
\begin{center}
\begin{tikzpicture}
	\draw[ultra thick, postaction = {decorate, decoration = {markings, mark = at position 0.75 with {\arrow{>}}}}] (-0.75,2) .. controls +(0.5,-1) and +(0,1) .. (-0.25,0.5);
		\draw[white, line width = 6pt] (-0.25,2) .. controls +(-0.5,-1) and +(0,1) .. (-0.75,0.5);
		\draw[ultra thick, postaction = {decorate, decoration = {markings, mark = at position 0.75 with {\arrow{>}}}}] (-0.25,2) .. controls +(-0.5,-1) and +(0,1) .. (-0.75,0.5);
		\draw[ultra thick, postaction = {decorate, decoration = {markings, mark = at position 0.75 with {\arrow{>}}}}] (-0.5,0.5) -- (-0.5,-0.5);
		\node[draw = black, fill = white, rectangle, minimum width = 1cm] at (-0.5,0.5) {$\alpha$};
		\node[left] at (-0.75,2) {$i$};
		\node[right] at (-0.25,2) {$j$};
		\node[left] at (-0.5,-0.5) {$k$};
	\node at (1,0.5) {$=$};
		\begin{scope}[shift = {(6.5,0)}]
		\node at (-3,0.5) {$\displaystyle \sum_{\beta\in\{1,\dots ,N_{ij}^k\}} R^+(i,j,k,\alpha,\beta)$};
		\draw[ultra thick, postaction = {decorate, decoration = {markings, mark = at position 0.75 with {\arrow{>}}}}] (-0.25,2) -- (-0.25,0.5);
		\draw[ultra thick, postaction = {decorate, decoration = {markings, mark = at position 0.75 with {\arrow{>}}}}] (-0.75,2) -- (-0.75,0.5);
		\draw[ultra thick, postaction = {decorate, decoration = {markings, mark = at position 0.75 with {\arrow{>}}}}] (-0.5,0.5) -- (-0.5,-0.5);
		\node[draw = black, fill = white, rectangle, minimum width = 1cm] at (-0.5,0.5) {$\beta$};
		\node[left] at (-0.75,2) {$i$};
		\node[right] at (-0.25,2) {$j$};
		\node[left] at (-0.5,-0.5) {$k$};
		\end{scope}
	\end{tikzpicture}
\end{center}
The inverse $R$-symbol $R^-(i,j,k,\alpha,\beta)$ is similarly defined to describe the inverse braiding.
We turn these $R$-symbols into 5-ary functions on the entire domain $D_\B$ in the same trivial way as before, namely
\[R^{\pm}(x_1,x_2,x_3,x_4,x_5) \defeq \begin{cases}R^\pm(x_1,x_2,x_3,x_4,x_5) & \text{ if }x_1,x_2,x_3\in \Irr(\B) \text{ and } x_4,x_5\in \N \\ 0 & \text{ otherwise}\end{cases}\]

Finally, we define 1-ary dimension functions, 1-ary Gauss sum functions (and their inverses), 1-ary dual functions, and 2-ary Kronecker delta functions as follows (respectively):
\[d(x) \defeq \begin{cases} \dim(x) & \text{ if } x\in \Irr(\B) \\ 0 & \text{ otherwise}\end{cases}\]
\[d^2(x) \defeq (d(x))^2\]
\[\sqrt{p_{\pm}}(x) \defeq\begin{cases} \left (\sum_{j\in\Irr(\B)} \theta_j^{\pm 1}\dim(j)^2\right )^{1/2} & \text{ if } x = * \\ 0 & \text{ otherwise }\end{cases}\]
\[\sqrt{p_\pm}^{-1}(x) \defeq \begin{cases}(\sqrt{p_{\pm}}(x))^{-1} & \text{ if } x = * \\ 0 & \text{ otherwise}\end{cases}\] 
\[\delta(x_1,x_2) \defeq \begin{cases}\delta_{x_1,x_2} & \text{ if } x_1,x_2\in \Irr(\B) \\ 0 & \text{ otherwise}\end{cases}\]

With all of this, we define our weighted constraint family $\F_\B$ to be
\[\F_\B \defeq \{F^{\pm}, G^{\pm}, BP, TT, VS^{\mathrm{left}}, VS^{\mathrm{right}}, R^{\pm}, d, d^2, \sqrt{p_{\pm}}, \sqrt{p_{\pm}}^{-1}, \delta\}.\]
We reiterate that all of this is computed independently of $M$, and can simply be considered as part of what it means to ``have the data" of $\B$.

We conclude our proof by describing how to encode a surgery presentation of $M$ into an instance $I_M$ of $\#\CSP(\F_\B)$.
In addition to the conveniences described above, we assume that the surgery link diagram $L$ is oriented, embedded in $\mathbb{R}\times [-1,2]$, and is given by plat closure of a braid word $b_1b_2\cdots b_n$ so that each crossing corresponding to $b_i$ lies in $\mathbb{R}\times (\frac{i-1}{n}, \frac{i}{n})$ and the only maxima or minima lie in $\mathbb{R}\times ([-1,0]\cup[1,2])$.  

In order to describe the variables that will be involved in the instance $I_M$ we want, we first describe a polynomial time algorithm to replace $L$ with a planar trivalent graph $\Gamma_L \subset S^2$:
\begin{enumerate}
	\item At each $\mathbb{R}\times \{\frac{i-1}{n}\}$, insert trivalent vertices to resolve the identity (see Figure \ref{fig:resolvingidentity}) so that there is a vertex directly adjacent to the crossing.
	\item Perform an $R$-move for each crossing, resulting in a trivalent graph in $\mathbb{R}^2 \subset S^2$ (after potentially scaling so the vertices lie in $\mathbb{R}\times \mathbb{Z}$).
\end{enumerate}

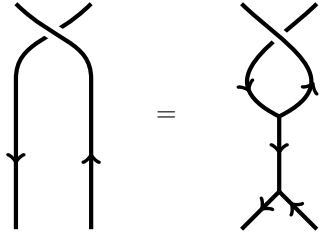
\begin{figure}
	\begin{tikzpicture}
		\draw[ultra thick, postaction = {decorate, decoration = {markings, mark = at position 0.75 with {\arrow{>}}}}] (0,2) .. controls +(-0.5,-0.5) and +(0,0.5) .. (-1,1) -- (-1,-1);
		\draw[white, line width = 6pt] (-1,2) .. controls +(0.5,-0.5) and +(0,0.5) .. (0,1);
		\draw[ultra thick, postaction = {decorate, decoration = {markings, mark = at position 0.75 with {\arrow{<}}}}] (-1,2) .. controls +(0.5,-0.5) and +(0,0.5) .. (0,1) -- (0,-1);
		\node at (1,0.5) {$=$};
		\begin{scope}[shift = {(3,0)}]
		\draw[ultra thick, postaction = {decorate, decoration = {markings, mark = at position 0.75 with {\arrow{>}}}}] (0,2) .. controls +(-0.5,-0.5) and +(-1,0.5) .. (-0.5,0.5);
		\draw[white, line width = 6pt] (-1,2) .. controls +(0.5,-0.5) and +(1,0.5) .. (-0.45,0.45);
		\draw[ultra thick, postaction = {decorate, decoration = {markings, mark = at position 0.75 with {\arrow{<}}}}] (-1,2) .. controls +(0.5,-0.5) and +(1,0.5) .. (-0.5,0.5);
		\draw[ultra thick, postaction = {decorate, decoration = {markings, mark = at position 0.5 with {\arrow{>}}}}] (-0.5,0.5) -- (-0.5,-0.5);
		\draw[ultra thick, postaction = {decorate, decoration = {markings, mark = at position 0.5 with {\arrow{>}}}}] (-0.5,-0.5) -- (-1,-1);
		\draw[ultra thick, postaction = {decorate, decoration = {markings, mark = at position 0.5 with {\arrow{<}}}}] (-0.5,-0.5) -- (0,-1);
		\end{scope}
	\end{tikzpicture}
\caption{Inserting trivalent vertices to resolve the identity}
\label{fig:resolvingidentity}
\end{figure}

We can now identify the tuple of variables (valued in the domain $D_\B$) that will be involved in our desired instance $I_M$.
Recall that $m$ is the number of components of $L$, $\sigma(L)$ is the signature, and $n$ is the number of crossings.
Let $\Gamma_0 = \Gamma_L, \Gamma_1, \dots, \Gamma_l = \emptyset$ be the sequence of graphs provided by Lemma \ref{lem:trivalentgraphalg}, and define $N_v$ to be the sum of the number of vertices in all of these graphs.
Similarly, define $N_e$ to be the sum of the number of edges in the graphs.
Then the variables of $I_M$ are described by a tuple
\[\x_M \defeq (x_0,x_1, \dots ,x_m, y_1, \dots ,y_{|\sigma(L)|}, z_1, \dots, z_n, u_1,\dots, u_{N_v}, w_1, \dots, w_{N_e}).\] 

To determine the functions involved in $I_M$, we simply need to keep appropriate account of the sequence of diagrammatic operations involved in taking $L$ to $\Gamma_L$ and then taking $\Gamma_L$ to $\emptyset$.
For each diagrammatic operation in the algorithms above, we will define a list of functions in $\F_\B$ which account for the contribution of those operations to $\tau_\B(M)$.

The first operations are those involved in step (1) of the process of taking $L$ to $\Gamma_L$, and involve the insertion of ``resolutions of the identity" at every crossing of $L$.
To account for these operations, we define a set $\mathrm{Init}$ that is a list of Kronecker delta functions, each of which pairs up edges which used to belong to the same link component before the operation.
For example, in Figure \ref{fig:resolvingidentity}, the top-right edge of the upper vertex and the bottom-right edge of the lower vertex were a part of the same link component before the operation, so we introduce a Kronecker delta function between the associated variables.

We then account for the operations in step (2) of the process of taking $L$ to $\Gamma_L$, all of which are $R$-moves.
Each operation happens locally on the diagram, so we define lists $R_i$ for $1\leq i \leq n$ that are given by $R^\pm(x_{i_1},x_{i_2},x_{i_3},x_{i_4},x_{i_5})$, where we use $+$ or $-$ depending on the strand that crosses over, $x_{i_1},x_{i_2},x_{i_3}$ are the edges in the trivalent vertex in the relevant order, $x_{i_4}$ is the labeling of the vertex before the $R$-move, and $x_{i_5}$ is the labeling of the vertex after the operation.

The next operations are given using the algorithm of Lemma \ref{lem:trivalentgraphalg}.
The algorithm provides an ordered list of operations $M_1,\dots, M_{l}$, where upon completion of the final operation $M_l$, the graph $\Gamma_l$ is empty.
Each operation here is local, so we need only consider the local changes when defining our list.
Consider operation $M_i$ in this sequence:

\begin{itemize}
\item If $M_i$ is a circle removal, define a list $C_i$ which contains the relevant $d$ function.
\item If $M_i$ is a tadpole trim, define a list $T_i$ which contains the relevant $TT$ function. 
\item If $M_i$ is a bubble pop, we define a list $B_i$ which contains the relevant $BP$ function. 
\item If $M_i$ is a vertex spiral, we define a list $V_i$ which contains the relevant $VS^{\mathrm{left}}$ or $VS^{\mathrm{right}}$. 
\item If $M_i$ is an $F$-move, we define a list $F_i$ which contains the relevant $F^\pm$ or $G^\pm$ functions. 
\item if $M_i$ is an orientation reversal where $w_{i_1}$ is the variable associated to the edge before the operation, and $w_{i_2}$ is the variable associated to the edge after the operation, then define a list $O_i$ which contains the Kronecker delta function $\delta(w_{i_1}^*, w_{i_2})$.
\end{itemize}

For each list, we append Kronecker delta functions $\delta$ on all edges or vertices which are held constant before and after performing the given operation. E.g. if the edge associated to $w_{22}$ will be held constant after an $F$-move, and will then be associated with $w_{100}$ in the trivalent graph associated to after the operation, then we introduce $\delta(w_{22},w_{100})$ to the list $F$ associated to the operation. We may also wish to swap the orientation of an edge, in which case we may do this at the cost of treating the coloring of the new edge as the dual of the old one. This does not require any functions to implement.

If $\sigma(L)\geq 0$, we then define the set $I_M$ by:
\begin{align*}
I_M = \{\sqrt{p_{+}}^{-1}(x_0),\sqrt{p_{-}}^{-1}(x_0),\sqrt{p_{+}}^{-1}(x_1), \dots ,\sqrt{p_{+}}^{-1}(x_m), \sqrt{p_{-}}^{-1}(x_1), \dots, \sqrt{p_{-}}^{-1}(x_m),\\ \sqrt{p_{-}}(y_1), \dots, \sqrt{p_-} (y_{|\sigma(L)|}), \sqrt{p_{+}}^{-1}(y_1),\dots, \sqrt{p_{+}}^{-1}(y_{|\sigma(L)|}), d^2(x_1), \dots d^2(x_m)\}\\
 \cup\mathrm{Init}\cup R_1\cup\cdots \cup R_n \cup \bigcup_{p_1,\dots ,p_6} C_{p_1}\cup T_{p_2}\cup B_{p_3} \cup V_{p_4} \cup F_{p_5} \cup O_{p_6}
\end{align*}
where the lists $C_{p_1}$, $T_{p_2}$, $B_{p_3}$, $V_{p_4}$, $F_{p_5}$, and $O_{p_6}$ each have functions depending on the $z$, $u$, and $w$ variables where relevant.
If $\sigma(L)<0$, then replace the elements $\sqrt{p_{-}}(y_1),\dots ,\sqrt{p_-}(y_{|\sigma(L)|})$ with the inverses $\sqrt{p_-}^{-1}(y_1),\dots,\sqrt{p_-}^{-1}(y_{|\sigma(L)|})$.
By virtue of Lemma \ref{lem:trivalentgraphalg}, the construction of $I_M$ can be carried out in polynomial time in the size of $M$.

We now explain why this choice of instance $I_M$ defines an output which computes the Reshetikhin-Turaev invariant $\tau_\B(M)$.
The $\sqrt{p_+}$ and $\sqrt{p_-}$ functions at the beginning implement the normalizing factor, and the $d^2$ functions implement the product of dimensions we see in the sum.
Notice that if at any point a coloring is not admissible, the term in $I_M$ is 0.
Now, it is clear that all circle removals, edge orientation reversals, bubble pops, and tadpole trims are correctly implemented, so we just need to check that the $F$-moves, $R$-moves, and vertex spirals are correct.
In the standard algorithm, when an $F$-move occurs, there is an additional variable introduced in the summation, along with the relevant $F$-symbol. This occurs here as well, since we are summing over all edges that occur throughout the algorithm, the new edge which is created in the $F$-move will contribute to the sum.
Similarly, we see that $R$-moves and vertex spirals are correctly implemented.
Note that also there are no extraneous variables in the sum since the algorithm guarantees that the graph will become empty.
\end{proof}

\section{Discussion}
\label{sec:disc}

\subsection{Unitarity}
\label{ss:unitarity}
We note that none of our results depend on the unitarity of the spherical fusion category $\C$ or the modular fusion category $\B$.
This is not surprising, since a choice of unitary structure is not necessary to define the TQFT invariants of closed $3$-manifolds from $\C$ or $\B$, and so, as far as exact calculation of invariants is concerned, such a choice will not affect any dichotomies.
Nevertheless, a unitary structure is certainly needed in order to do topological quantum computation with the TQFT determined by $\C$ or $\B$, since, for such applications, one needs the TQFT to be unitary.
\emph{A priori}, a specific choice of unitary structure might affect the $\BQP$-universality of braiding (with or without adaptive measurement), although we expect that \emph{a posteriori}, this will not be the case.

\subsection{TQFTs in other dimensions}
\label{ss:dimension}
We furthermore note that the same strategy we used for the proof of Theorem \ref{thm:main}(a) should work more generally to prove that any fully-extended $(d+1)$-dimensional TQFT in any dimension will satisfy a similar dichotomy involving its invariants of closed $(d+1)$-dimensional manifolds, so long as the TQFT is defined using a state-sum formula based on finite combinatorial-algebraic data.
In particular, similar dichotomies should be possible for $(3+1)$-dimensional TQFTs based on spherical fusion 2-categories \cite{DouglasReutter} or lattice gauge theories based on finite groups (sometimes called Dijkgraaf-Witten theories) in arbitrary dimension.

\subsection{Alternative proof strategies}
\label{ss:alternative}
Building on the previous point, one might try to give an alternative proof of Theorem \ref{thm:main}(b) by using the $(3+1)$-dimensional Crane-Yetter TQFT based on the modular fusion category $\B$ \cite{CraneYetter}.
To put it more carefully, it is known that the Reshetikhin-Turaev invariant $\tau_\B(M)$ can be computed by choosing a triangulated 4-manifold $Y$ whose boundary is $\partial Y = M$, and computing an appropriate state-sum invariant of $Y$ (similar to the Turaev-Viro invariant of a triangulated 3-manifold) \cite{CraneYetterKauffman}.
So one could try to prove a dichotomy for the $(2+1)$-dimensional surgery-invariant case of Reshetikhin-Turaev in a simpler way by instead proving a dichotomy for the $(3+1)$-dimensional triangulation-invariant case of Crane-Yetter.
Accomplishing this would require showing that given a surgery diagram for a $3$-manifold $M$, one can build a triangulated 4-manifold $Y$ with $\partial Y = M$ in polynomial time.

In the opposite direction, it is known that for a spherical fusion category $\C$, $|M|_\C = \tau_{\mathcal{Z}(\C)}(M)$, where $\mathcal{Z}(\C)$ is the Drinfeld center of $\C$.  If one were able to efficiently convert a triangulation of a 3-manifold $M$ into a surgery presentation of the same manifold, then part (b) of Theorem \ref{thm:main} (or Conjecture \ref{con:main}) would immediately imply part (a) of the same.

\subsection{Towards Conjecture \ref{con:main}}
\label{ss:improvements}
The current gap between Theorem \ref{thm:main} and Conjecture \ref{con:main} is explained by a rather simple and undesirable property of our proof of the former: our reductions $M \mapsto I_M$ are not ``surjective" from 3-manifold encodings to instances of $\#\CSP(\F_\C)$ or $\#\CSP(\F_\B)$.
For example, it would be consistent with our results for there to exist a spherical fusion category $\C$ such that $|M|_\C$ is computable in polynomial-time, and yet, the problem $\#\CSP(\F_\C)$ is still $\shP$-hard (although we consider this unlikely).

To establish an outright dichotomy theorem for TQFT invariants via Cai and Chen's Theorem \ref{thm:dichotomy}, we would need to arrange our choices of $\F_\C$ and $\F_\B$ with more care so that every instance $I$ of $\#\CSP(\F_\C)$ or $\#\CSP(\F_\B)$ that is not of the form $I_M$ satisfies two properties: first, it can be identified in polynomial time as an instance that is not of the form $I_M$, and, second, $Z(I)$ can be computed in polynomial time.
This seems difficult to arrange, as it is not clear how to choose the constraint families $\F_\C$ and $\F_\B$ so that instances can ``self-report" as not being of the form $I_M$. 

It is instructive to compare TQFT invariants with ``holant problems" as defined in \cite{CaiETAL-holantdef} and inspired by the ``holographic reductions" of \cite{Valiant-holographic}.
Holant problems are a kind of generalization of counting CSPs that impose more structure on the way in which the individual functions comprising an instance are ``wired together."
Intuitively, an instance of $\#\CSP(\F)$ has no locality constraints on its variables, other than that the constraint functions $f \in \F$ have bounded arity (assuming $\F$ is finite).
An instance of a holant problem, on the other hand, has a set of variables that are determined by the \emph{edges of a graph} with constraint functions assigned to the \emph{vertices}.
The Turaev-Viro-Barrett-Westbury invariant of closed 3-manifolds determined by a spherical fusion category $\C$ can be seen as generalization of this idea, with variables assigned to both the edges \emph{and} faces of a 3-dimensional triangulation, and constraints assigned to the tetrahedra.
We expect it should be possible to formulate TVBW invariants of triangulated 3-manifolds directly as instances of holant problems using a similar construction as in the proof of Theorem \ref{thm:main}(a).
Of course, even if one could achieve this, such a reduction from TVBW invariants to holant problems would---\emph{a priori}---suffer in the same way as our current reduction to $\#\CSP(\F_\C)$.
Thus, it seems likely that proving Conjecture \ref{con:main} will require substantially new ideas.
Nevertheless, it appears that there could be much to gain by attempting to import what has been learned about holant problem dichotomies to TQFT invariants of 3-manifolds.


\end{document}